\documentclass[psamsfonts,10pt]{amsart}
\usepackage[english]{babel}
\usepackage{mathpazo}

\usepackage{overpic}
\usepackage{amsmath,amsthm,amssymb,amscd,amsfonts,amsbsy}
\usepackage{color}
\usepackage{hyperref,url}
\usepackage{float}
\usepackage{hyperref} 
\usepackage{enumerate}
\usepackage{upgreek}
\usepackage{enumitem}   
\usepackage{mathrsfs}  
\usepackage{mathrsfs}  
\usepackage{amsthm}
\usepackage{xfrac}
\usepackage{amsmath}
\usepackage{amsfonts}
\usepackage{amssymb}
\usepackage{amsthm}
\usepackage{amsmath}
\usepackage{amsfonts}
\usepackage{amssymb}
\usepackage{mathtools}
\usepackage{mathrsfs}
\usepackage{comment}
\usepackage{import}

\usepackage{amsmath, amssymb, graphics, setspace}
\usepackage[utf8]{inputenc}
\usepackage{listings,xcolor}

\newcommand{\mathsym}[1]{{}}
\newcommand{\unicode}[1]{{}}

\usepackage[normalem]{ulem}

\newcommand{\R}{\ensuremath{\mathbb{R}}}
\newcommand{\CC}{\mathcal{C}}
\newcommand{\C}{\mathbb{C}}
\newcommand{\Z}{\mathbb{Z}}

\newcommand{\ov}{\overline}

\newcommand{\vertiii}[1]{{\left\vert\kern-0.25ex\left\vert\kern-0.25ex\left\vert #1 
		\right\vert\kern-0.25ex\right\vert\kern-0.25ex\right\vert}}

\newcommand{\T}{\theta}

\newcommand{\de}{\delta}
\newcommand{\e}{\varepsilon}

\newcommand{\CL}{\mathcal{L}}

\newcommand{\ce}{c^{*}}
\newcommand{\sgn}{\mathrm{sign}}

\newcommand{\ave}[1]{\langle #1  \rangle}
\newcommand{\normindice}[2]{\left\Vert #1\right\Vert_{#2}}

\newcommand{\deltaind}[1]{\delta^{(#1)}}
\newcommand{\muind}[1]{\mu^{(#1)}}
\newcommand{\rhoind}[1]{\rho^{(#1)}}
\newcommand{\tauind}[1]{\cur^{(#1)}}
\newcommand{\tauphiind}[1]{\cur_{\phi}^{(#1)}}

\newcommand{\Aind}[1]{A^{(#1)}}
\newcommand{\eind}[1]{e^{(#1)}}

\newcommand{\pI}{\partial_I}
\newcommand{\pphi}{\partial_{\phi}}

\newcommand{\normnind}[2]{\vertiii{#1}_{#2}}

\newcommand\blfootnote[1]{%
	\begingroup
	\renewcommand\thefootnote{}\footnote{#1}%
	\addtocounter{footnote}{-1}%
	\endgroup
}

\newcommand\phantomarrow[2]{%
	\setbox0=\hbox{$\displaystyle #1\to$}%
	\hbox to \wd0{%
		$#2\mapstochar
		\cleaders\hbox{$\mkern-1mu\relbar\mkern-3mu$}\hfill
		\mkern-7mu\rightarrow$}%
	\,}

\usepackage[toc,page]{appendix}

\newtheorem {theorem} {Theorem}

\newtheorem {proposition} [theorem]{Proposition}

\newtheorem {lemma}  [theorem]{Lemma}

\newtheorem {remark}{Remark}

\newtheorem* {lemmaiterativo} {Iterative Lemma}
\newtheorem* {maintheoremA} {Theorem A}
\newtheorem* {maintheoremB} {Theorem B}

\def\R{\mathbb R}
\def \d {\mathrm{d}}

\def \tbo {\bar{t}_0}
\def \ybo {\bar{y}_0}
\def \Ebo {\bar{E}_0}
\def \fbto {\bar{f}_{t_0}}
\def \fbEo  {\bar{f}_{E_0}}

\def \Eot {{E}_0^{*}}
\def \Iot {{I}_0^{*}}
\def \Put {\tilde{P}_1}
\def \Pdt {\tilde{P}_2}
\def \rot {\tilde{ p}}

\def \Ffm {\mathcal{F}^{+}}
\def \Dfm {\mathcal{D}^{+}}
\def \tb {\bar{\tau}}
\def \tm {\cur_{+}}
\def \am {a_+}
\def \bm {b_+}
\def \t {\tau}
\def \tmest {\tau^{+}_*}

\def \etap {\eta^{+}}
\def \tme {\tau^{-}}
\def \tmeest {\tau^{-}_*}

\def \rotzm {\tilde{p}_0^{+}}
\def \robzm {\bar{p}_0^{+}}

\def \partialto {\partial_{t_0}}
\def \partialyo {\partial_{y_0}}

\def \Sp {\Sigma^{+}}
\def \Sn {\Sigma^{-}}

\def \tmt {\tau^{+}}

\def \To  {\mathbb{T}}
\def \Lome {\mathcal{L}_{\omega}}
\def \Omu {\Omega_+^{-1}}
\def \D {\Delta}
\def \CR {\mathcal{R}}
\def \Omegau {(\Omega^{(1)})^{-1}}
\def \Omuu {(\Omega_+^{(1)})^{-1}}

\def \xx {\mathrm{x}}
\def \yy {\mathrm{y}}

\def \cur {\varphi}

\title[Boundedness of solutions of a forced discontinuous oscillator]
{On the boundedness of solutions of a forced discontinuous oscillator}

\author[T. M-Seara]
{Tere M-Seara$^{1,2,3}$}

\author[L. V. M. F. Silva]
{$ ^{*} $Luan V. M. F. Silva$^4$}

\author[J. Villanueva]
       {Jordi Villanueva$^1$}

\address{$^1$ Departament de Matemàtiques, Universitat Politècnica de Catalunya, Diagonal 647, 08028 Barcelona, Spain}

\address{$^2$ Institut de Matemàtiques de la UPC - Barcelona Tech (IMTech), Pau Gargallo 14, 08028 Barcelona, Spain.
}

\address{$^3$ Centre de Recerca Matemàtica, Edifici C, Campus Bellaterra, 08193 Bellaterra, Spain.
}

\address{$^4$ Departamento de Matem\'{a}tica, Universidade
	Estadual de Campinas, Rua S\'{e}rgio Baruque de Holanda, 651, Cidade
	Universit\'{a}ria Zeferino Vaz, 13083--859, Campinas, SP, Brazil}

\allowdisplaybreaks
\begin{document}

\noindent

\noindent

\blfootnote{$^*$Corresponding author: Luan V. M. F. Silva, luanmattos@ime.unicamp.br} 

\maketitle

\begin{abstract} 
We study the global boundedness of the solutions of a non-smooth forced oscillator with a periodic and real analytic forcing. We show that the impact map associated with this discontinuous equation becomes a real analytic and exact symplectic map when written in suitable canonical coordinates. By an accurate study of the behaviour of the map for large amplitudes and by employing a parametrization KAM theorem, we show that the periodic solutions of the unperturbed oscillator persist as two-dimensional tori under conditions that depend on the Diophantine conditions of the frequency, but are independent on both the amplitude of the orbit and of the specific value of the frequency. This allows the construction of a sequence of nested invariant tori of increasing amplitude that confine the solutions within them, ensuring their boundedness. The same construction may be useful to address such persistence problem for a larger class of non-smooth forced oscillators.	
\end{abstract}

\section{Introduction}\label{introduction}
The study of the boundedness of solutions in Duffing-type equations has been a long-standing problem in dynamical systems, dating back to the seminal works of Littlewood in the 1960's~\cite{Littlewood1966A,Littlewood1966B}. In the referred works, Littlewood proved the existence of unbounded solutions to the equation 
\begin{equation}\label{duffing}
	\ddot{x}+g(x)=p(t),
\end{equation}
with $ p(t) $ being bounded and periodic and for certain function $ g(x) $, usually called saturation function. However, in contrast to this unbounded behavior, Littlewood~\cite{Littlewood1968} raised the question of whether all solutions of~\eqref{duffing} would indeed be bounded, prompting further investigations into different conditions on $ p(t) $ and $ g(x) $ beyond the aforementioned ones. It was Morris~\cite{Morris1976} who first made a significant contribution by providing the initial example, solely under the assumption that $ p(t) $ is a periodic continuous function and $ g(x) = 2x^{3}$. Some years later, Dieckerhoff and Zehnder~\cite{zehnder} further expanded upon the Morris result by investigating the equation
\[
	\ddot{x}+x^{2n+1}+\sum_{j=0}^{2n}x^j p_j(t)=0, \quad n\geq 1,
\]
with $p_j$ denoting periodic $\CC^{\infty} $ functions.

As noted in~\cite{Li2001}, in 1998, during a presentation at the Academia Sinica, Ortega~\cite{Ortega1998Talk} introduced the idea of investigating the global stability of~\eqref{duffing}, assuming the presence of a bounded saturation function $ g $. He specifically suggested to take $g(x) = \arctan(x)$ and aimed to identify the conditions on $p(t)$ that would ensure this desired result. However, due to the bounded saturation function causing a subtle twist at infinity, applying the standard twist map theorem becomes challenging in obtaining global conclusions. In this circumstance, Li~\cite{Li2001} was the first to establish a result based on Ortega's proposal, assuming that $p(t)$ is a $\CC^{\infty}$ periodic function with zero average. This initial result was further enhanced by Wang in~\cite{Wang2006}, where the regularity requirement on $p(t)$ was relaxed to $\CC^{5}$ while also considering a small condition on its average.

As a limit scenario within the framework proposed by Ortega~\cite{Ortega1998Talk}, our investigation focuses on examining the boundedness of solutions for a non-smooth forced oscillator described by the equation
\begin{equation}\label{eq1}
	\ddot{x} + \sgn(x) = \e \; p(t),
\end{equation}
with $\sgn$ representing the standard sign function, $\e \geq 0$ being a small real parameter, and the forcing term being a real analytic and $2\pi$-periodic function
\begin{equation}\label{eq:p}
	p(t) = a_0 + \sum_{k\geq 1}(a_k \cos(kt) + b_k\sin(kt)),
\end{equation}
so that it is an analytic function of $\mathbb{T}=\R/2\pi\Z$. 

In the unperturbed scenario $\e=0$, all solutions of equation~\eqref{eq1} are periodic (see Fig.~\ref{fig1}). This implies that, in particular, all remain bounded in the $(x,\dot x)$-plane for $\e=0$. A natural question that arises from this observation is whether the bounded nature of all solutions remains true when subjected to small perturbations.

\begin{figure}[!t]
	\begin{center}
		\begin{overpic}[scale=0.5]{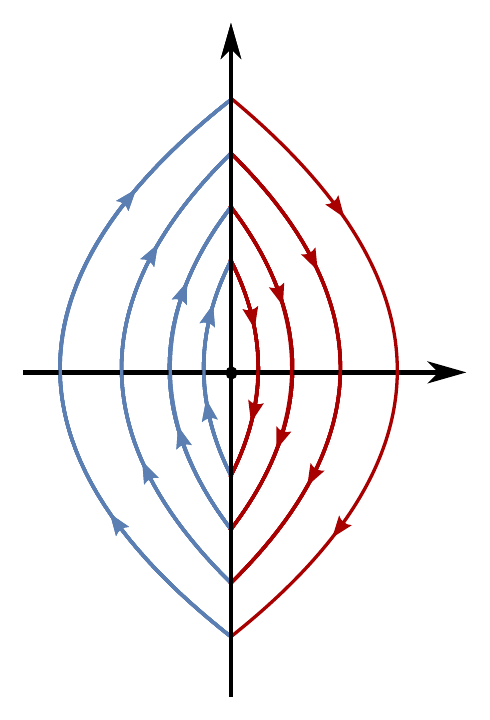}
		\put(66,47){$ x $}
			\put(31.5,99){$ y $}
		\end{overpic}		
	\end{center}
	\caption{Phase portrait of~\eqref{eq1} for the unperturbed case ($ \e=0 $).}
	\label{fig1}
	\smallskip
\end{figure}

Significant progress has been made regarding the analysis of equation~\eqref{eq1} in the non-pertur\-ba\-ti\-ve context $\e=1$. For example, in~\cite{Ortega2019}, Engui\c{c}a and Ortega proved that equation~\eqref{eq1} has infinitely many bounded solutions in the $(x,\dot{x})$-plane (of arbitrarily large amplitude) if the function $p(t)$ is, not necessarily periodic, but bounded and satisfies that the limit
\[
\bar{p}=\lim_{T\to+\infty}\frac{1}{T}\int_{t}^{t+T}p(s)ds
\]
exists uniformly in $t\in\R $, with $|\bar p|<1$. Recently, Novaes and Silva~\cite{Novaes2023} proved that all solutions remain bounded if $p(t)$ is a Lebesgue integrable periodic function with vanishing average.

Further investigations related to the perturbative equation~\eqref{eq1} can also be found in the research literature. For instance, Jacquemard and Teixeira~\cite{Jacquemard2012} studied the persistence of periodic orbits, while Burra and Zanolin~\cite{Burra2020} discovered the existence of chaos near the origin under the assumption that $p(t)$ is a piecewise constant and periodic function. The significance of equation~\eqref{eq1} extends beyond theoretical interest, finding practical applications in the field of electronics. In this context, the presence of the $\sgn$ function can be interpreted as a representation of a non-smooth oscillator operating in conjunction with a relay, as emphasized in~\cite{Jacquemard2012}. 

The main result of this work is \hyperref[TA]{Theorem A}, where we prove the boundedness of all solutions of~\eqref{eq1} by assuming that $p(t)$ is a real analytic and $2\pi$-periodic function, regardless of the value of the average of $p(t)$. The main idea behind the proof of \hyperref[TA]{Theorem A} is to determine the existence, for any given $\e>0$ sufficiently small, of an infinite family of nested two-dimensional invariant tori in the extended phase space $(t,x,\dot{x})\in\To\times\R^2$. The projection of these tori onto the $(x,\dot{x})$ plane surrounds the origin $ (0,0) $ and their amplitude becomes unbounded as we move away from it. Therefore, each of these tori perpetually confines all solutions with initial conditions enclosed within it.

One standard approach to find these tori is to perform changes of variables to the discontinuous differential equations to overcome their lack of regularity, as described in~\cite{KunzeArt, Levi1991} and references therein. This approach leads to a finite diferentiable perturbed system which allows the application of KAM theory to its stroboscopic Poincar\'{e} map. Subsequently, variants of Moser's twist map theorem~\cite{Moser1962} are employed to identify closed invariant curves for this map. By doing so, the existence of the tori of the original system can be determined. However, for equation~\eqref{eq1}, the application of such coordinate changes leads to a new system for which it is not apparent that we can establish a uniform upper bound $ \e^{\ast}>0 $ in a way that all solutions of~\eqref{eq1} would be bounded if $0<\e<\e^{\ast}$. This is primarily due to the fact that the twist condition of the system obtained after applying the different coordinate changes to~\eqref{eq1} diminishes significantly as we progressively distance ourselves from the origin. Consequently, in order to be able to directly apply this procedure to~\eqref{eq1}, we must show that, after transformation, the size of the resulting perturbation strongly decreases with this distance. Showing this strong decrease for this system does not appear to be an easy task, with the aggravating factor that this must be done in terms of a $\CC^{r}$-topology.

Our method to prove \hyperref[TA]{Theorem A} involves analyzing the so called impact map associated with the equation (see Fig.~\ref{fig2} and equation~\eqref{eq:Pe}) in appropriate canonical time-energy coordinates (see equation~\eqref{ImpactAngleEnergy}). These coordinates arise in a natural way after rewriting~\eqref{eq1} as an autonomous Hamiltonian system (see equation~\eqref{H2}). In these coordinates, the impact map is a real analytic and exact symplectic map of the annulus (see Proposition~\ref{ExactSympleticGeneral}), which turns out to be a perturbation of an integrable twist map. Although the impact map in these canonical coordinates possesses favorable properties regarding to the study of the persistence of the invariant curves of the integrable approximation $\e=0$, at first sight the same drawbacks outlined above regarding to the standard approximation persist for this analytic map. Explicitly, if we apply to it a KAM theorem for the persistence of invariant curves of analytic exact symplectic maps, it is easy to establish the persistence, up to a certain value of $\e>0$, of any given unperturbed curve having a Diophantine rotation number. However, it is not clear that we can simultaneously establish the persistence of invariant curves of arbitrarily large amplitude.
\begin{remark}
  It is worth noting that the process we use to introduce these canonical time-energy coordinates can easily be extended, at least locally, to a very wide range of non-smooth forced oscillators. When these coordinates can be introduced, we have that the analyticity, the exact symplectic character and the perturbative form of the impact map follow at zero cost simply from the structure of the initial system. This means that if we are only interested in proving a local persistence result of invariant curves, such a result can usually be reached through a very moderate amount of work. The fact that proving \hyperref[TA]{Theorem A} requires performing the elaborate process described below is because in our case we intend to prove a global persistence result.
\end{remark}
In order to prove \hyperref[TA]{Theorem A}, it has been necessary to perform a precise analysis of the behaviour of the impact map for large amplitudes. We have identified its dominant terms and we have controlled the asymptotic size of the remaining ones. As a consequence of this analysis we have been able to carry out the following construction. We consider a specific invariant curve of the unperturbed impact map with sufficiently large amplitude and with frequency that satisfies the Diophantine estimates~\eqref{diophantine}, for some pair $ (\gamma, \nu)$. This curve is characterized by a large enough (in modulus) value of the action variable of the map, that we call $E_0^\ast$. Next, we localize the impact map around $E_0^\ast$ and perform an appropriate conjugation (scaling) on this localized map, depending on $E_0^\ast$. In this way, we construct another analytic and exact symplectic map of the annulus. This map, which we can refer to as $F_{E_0^\ast}$, is different action by action and closely describes the dynamics around the selected curve. Regardless of the amplitude of the curve, for $F_{E_0^\ast}$ we have both that the twist condition is always of order one and that the size of the perturbative terms of the corresponding integrable approximation can also be bounded uniformly for all the actions (see Proposition~\ref{P1}). Then, we study the persistence of invariant curves of the impact map of arbitrarily large amplitude by applying a parametrization KAM theorem for analytic and exact symplectic twist maps of the annulus (see \hyperref[TB]{Theorem B}) to these maps $F_{E_0^\ast}$. Using this approach we do not have to worry about either the smallness of the twist condition or the control of the size of the non-integrable part as the amplitude increases. This approach of addressing the persistence of a family of tori of a specific system in terms of a parametric family of systems labelled by their frequency or frequency vector (here $E_0^\ast$ determines the frequency) has been used previously in other contexts of KAM theory (see e.g.~\cite{JorbaV97A,JorbaV97B}).

The final conclusion of this approach is that there exists $\e^{\ast}>0$ such that, for any $0<\e<\e^{\ast}$, there are infinitely many invariant curves of the unperturbed impact map that survive the perturbation. The amplitude of these curves tends to infinite and so do their frequencies. But the Diophantine constant $\gamma$ of all these frequencies can be bounded from below by the same value. Actually, $\e^{\ast}$ only depends on this uniformly lower bound of $\gamma$ and on the periodic function $p(t)$. By integrating by~\eqref{eq1} these invariant curves of the impact map and expressed in terms of $(t,x,\dot x)$, we construct the infinite collection of nested invariant tori $\{\mathcal{S}_{j}^{\e}\}_{j=0}^{\infty}$ for the equation~\eqref{eq1} in the extended phase space, thereby ensuring the boundedness of all its solutions, as long as $\e$ remains within this specified range.

For the sake of completeness, in this paper we have not only included the statement of the parametrization KAM \hyperref[TB]{Theorem B}, but we have also included its proof. Firstly, we have done this because \hyperref[TB]{Theorem B} is perfectly suited for application to the local persistence of invariant curves of analytic and close to integrable exact symplectic twist maps of the annulus. Other results in the literature on parametrization KAM theory are sometimes developed for systems with a more general structure and, therefore, need to be interpreted appropriately in order to be applied to the present context. Furthermore, we also think that its rather schematic but comprehensive proof may be useful for the interested reader who is not familiar with this kind of KAM results. 

 Finally, although we have addressed \hyperref[TA]{Theorem A} for a real analytic and periodic function $p(t)$, the same approach should be valid when $p(t)$ is only finitely differentiable. The main modifications with respect to the analytic case should appear when performing the estimates of the impact map for large amplitudes, since the size of the derivatives involved should be controlled without using Cauchy estimates. Of course, we must also consider using an appropriate version of the twist theorem for finitely differentiable exact symplectic maps of the annulus.

 The structure of this paper is outlined as follows. Section~\ref{mainresults} provides the main result of the paper, \hyperref[TA]{Theorem A}, as well as an overview of the different results (Propositions~\ref{ExactSympleticGeneral}, \ref{P}, and~\ref{P1} and \hyperref[TB]{Theorem B} whose sequential application leads to \hyperref[TA]{Theorem A}. In Section~\ref{ParameterizationKam}, we prove \hyperref[TB]{Theorem B}, which is a KAM result (based in the parametrization method) for the persistence of quasi-periodic invariant curves of an analytic and exact symplectic twist map of the annulus. Section~\ref{propP} examines the properties and bounds of the impact map associated with equation~\eqref{eq1} and proves Proposition~\ref{P}. Section~\ref{propP1} focuses on ensuring the exact symplectic properties and appropriate bounds of the localized and scaled impact map in suitable coordinates, as stated in Proposition~\ref{P1}. Finally, Section~\ref{sympletic} offers a comprehensive result concerning the exact sympletic character for the impact map written in suitable variables (see Proposition~\ref{ExactSympleticGeneral}).

\section{Main results}\label{mainresults}
The main result of this work concerns to the global stability of the solutions of~\eqref{eq1}, as follows.
\begin{maintheoremA}\label{TA}
There exists $ \e^{\ast}>0 $, depending on $ p $, such that if $ 0\leq\e<\e^{\ast} $, then all solutions of~\eqref{eq1} are bounded. 
\end{maintheoremA}
In this section we prove \hyperref[TA]{Theorem A} by the recurrent application of the results and constructions presented below. The proof of each technical result necessary to establish the theorem is postponed to a later section of the paper, as is indicated in each case.

To prove \hyperref[TA]{Theorem A}, we first check that a suitable impact map associated to~\eqref{eq1}, in good coordinates, is an analytic and exact-symplectic map which turns out to be a perturbation of an integrable twist map. Here, we briefly explain the formal construction of the impact map and we postpone the quantitative details to Section~\ref{propP}.

\begin{figure}[!t]
	\begin{center}
			\begin{overpic}[scale=0.5]{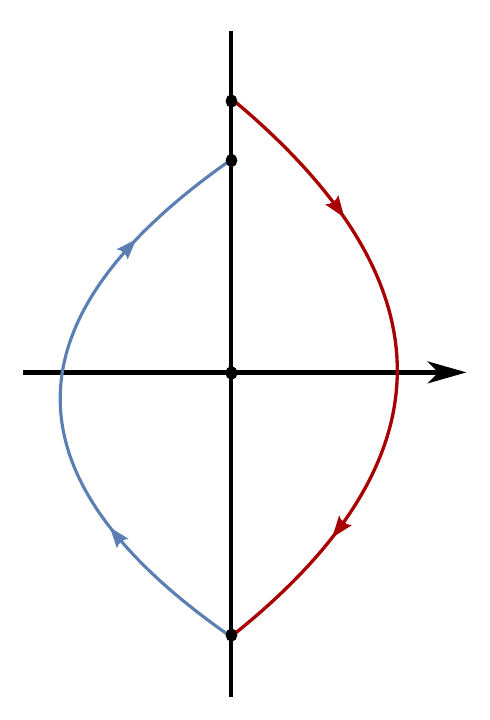}
			\put(28,97){{ $ \Sigma $}}
			\put(34.5,85.5){$(t_0,0,y_0)$}
			\put(34.5,9){$\mathcal{P}_{\e}^{+}(t_0,0,y_0)=(t_1,0,y_1)$}
			\put(-38,78.3){$\mathcal{P}_{\e}(t_0,0,y_0)=\mathcal{P}_{\e}^{-}(t_1,0,y_1)$}
			\put(55,55){{ $ \Phi_r^{\tau}$}}
			\put(2,55){{ $ \Phi_l^{\tau}$}}
			\put(66.5,46.8){$ x $}
		\end{overpic}		
	\end{center}
	\caption{Impact map $  \mathcal{P}_{\e}= \mathcal{P}_{\e}^{-}\circ\mathcal{P}_{\e}^{+}$.}
	\label{fig2}
	\smallskip
\end{figure}

If we introduce $ y=\dot{x} $, the differential equation~\eqref{eq1} can be seen as the vector field 
\begin{equation}\label{VF}
	\left\{\begin{aligned}
		\dot{t}&=1,\\ 
		\dot{x}&=y,\\
		\dot{y}&=-\sgn(x)+\e\;p(t),
	\end{aligned}\right.
\end{equation}
in the extended phase space $\mathbb{T}\times \mathbb{R}^2$, where the presence of the function $ \sgn $ configures the plane $\Sigma:=\{(t,x,y)\in\R^3\,:\,x=0\} $ as a region of discontinuity of~\eqref{VF}. The subsets of $\Sigma$
\begin{equation}\label{SpSn}
\Sp:=\{(t,0,y)\in\Sigma\,:\,y>0\} \quad \mbox{and} \quad \Sn:=\{(t,0,y)\in\Sigma\,:\,y<0\}
\end{equation}
play an important role in defining the domain of the half impact maps.  We denote the solutions of~\eqref{VF} with initial condition $ (t_0,x_0,y_0) $ and $ (t_1,x_1,y_1) $,  if $ x_0>0 $ and $ x_1 <0 $, respectively, by
\begin{equation}\label{solp}
\Phi^{\tau}_r(t_0,x_0,y_0;\e)=(t_r^{\tau}(t_0,x_0,y_0;\e),x_r^{\tau} (t_0,x_0,y_0;\e),y_r^{\tau}(t_0,x_0,y_0;\e)),
\end{equation}
and 
\begin{equation}\label{soln}
\Phi^{\tau}_l(t_1,x_1,y_1;\e)=	(t_l^{\tau}(t_1,x_1,y_1;\e),x_l^{\tau}(t_1,x_1,y_1;\e),y_l^{\tau}(t_1,x_1,y_1;\e)),
\end{equation}
respectively, where $ t_r^{\tau}(t_0,x_0,y_0;\e)=t_0+\tau $ and $ t_l^{\tau}(t_1,x_1,y_1;\e)=t_1+\tau $. For points in $ \Sigma^{+} $, we will use $ \Phi^{\tau}_r $ as the flow in $ \Sigma^{+} $ ``points to the right''. Analogously, for points in $ \Sigma^{-} $  we will use $ \Phi^{\tau}_l $ as the flow in $ \Sigma^{-} $ ``points to the left''. By adopting the Filippov convention (see~\cite{Filippov1988} for more detailed information) for the solutions of~\eqref{VF}, we find out that such solutions are obtained by the concatenation of $\Phi^{\tau}_r$ with $  \Phi^{\tau}_l $ along the crossing region $ \Sigma $. This construction ensures the uniqueness and continuity of solutions. For $ (t_0,0,y_0)\in\Sp $, we denote by $ \tmt (t_0,y_0;\e) $ the smallest positive time such that
\begin{equation}\label{eqaux1}
	x_r^{\tmt (t_0,y_0;\e) }(t_0,0,y_0;\e)=0.
\end{equation}
We note that, as~\eqref{VF} depends periodically in time, $ \tmt (t_0,y_0;\e) $ is $2\pi$-periodic in $t_0$. The half positive impact map is given by
\begin{equation}\label{pmais}
	\mathcal{P}^{+}_{\e}:	(t_0,0,y_0)\in\Sigma^+\mapsto(t_0+\tmt (t_0,y_0;\e),0,y_r^{\tmt(t_0,y_0;\e) }(t_0,0,y_0;\e))\in\Sigma^-.
\end{equation}
The initial condition $ (t_1,0,y_1)\in\Sigma^- $, follows the left flow given by~\eqref{soln}. We consider $  \tme (t_1,y_1;\e)  $ as the smallest positive time such that
\begin{equation}\label{eqaux2}
 x_l^{\tme (t_1,y_1;\e) }(t_1,0,y_1;\e)=0 ,
\end{equation}
with $  \tme (t_1,y_1;\e)  $ being $2\pi$ periodic in $t_1$. Then, the half negative impact map is given by
\begin{equation}\label{pmenos}
	\mathcal{P}^{-}_{\e}:	(t_1,0,y_1)\in\Sigma^-\mapsto(t_1+\tme (t_1,y_1;\e),0,y_l^{\tme (t_1,y_1;\e) }(t_1,0,y_1;\e))\in\Sigma^+.
\end{equation}
Thus, the complete impact map for equation~\eqref{eq1} is given by the composition of the negative with the positive half impact map, respectively, i.e.,
\begin{equation*}
	\mathcal{P}_{\e}:(t_0,0,y_0)\in\Sigma^+\mapsto (\overline{t}_0,0,\overline{y}_0)=(\mathcal{P}^{-}_{\e}\circ\mathcal{P}^{+}_{\e})(t_0,0,y_0)\in\Sigma^+,
\end{equation*}
with $2\pi$-periodic dependence in $t_0$, so that it can be read as a map of the annulus $\mathbb{T}\times\{0\}\times\R^+$ (see Fig.~\ref{fig2}).

When $ \e=0 $, it is easy to check that
\[
\mathcal{P}_0(t_0,0,y_0)=(t_0+\alpha(y_0),0,y_0),
\]
with $ \alpha(y_0)=4y_0$ is the period of the periodic solution of the unperturbed system~\eqref{eq1} with initial condition $ (t_0,0,y_0) $. Moreover, given $y_0>0$ and for sufficiently small $ \e \geq0$, the trajectories of~\eqref{eq1} starting in $ \Sp $ cross $ \Sp $ again, then $\mathcal{P}_{\e} $ is well defined and also analytic. Since $ \alpha'(y_0)=4>0 $, for all $y_0$, it follows that $ \mathcal{P}_0 $ is an integrable twist map of the annulus and any circle of the form $\mathbb{T}\times\{0\}\times\{y_0\}$, for all $y_0>0$, is an invariant curve of $ \mathcal{P}_0 $ with frequency $ \alpha(y_0)$ (see Fig.~\ref{fig3}). For those $y_0$ such that its associated rotation number $\alpha(y_0)/2\pi$ is irrational, the motion on this curve is quasi-periodic. In Lemmas~\ref{lemataup} and~\ref{lemataumenos} we provide asymptotic expressions for the times of impact $ \tmt $ and $ \tme $, respectively, for large (enough) amplitudes and whenever $ 0\leq \e <\e^{+}_*$ and $ 0\leq \e <\e^{-}_*$, respectively, with $ \e^{+}_* $ and $ \e^{-}_* $ depending only on the function $p  $. These results allow us to control how far $ \mathcal{P}_{\e} $ is to be an integrable map.

\begin{figure}[!t]
	\begin{center}
	\begin{overpic}[scale=0.35]{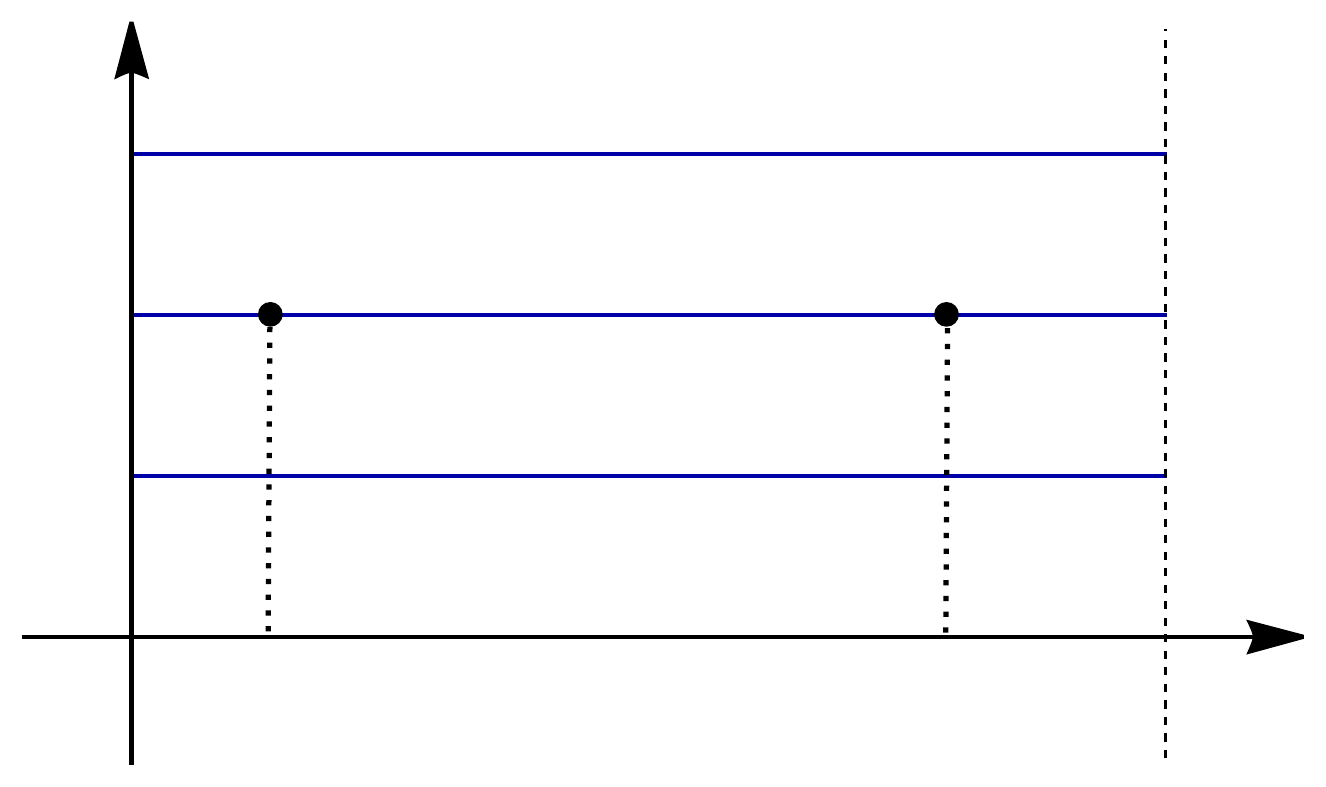}					
						\put(45,55){{ $ \Sigma^{+} $}}
						\put(13,38){$(t_0,0,y_0)$}
					\put(63,38){$\mathcal{P}_{0}(t_0,0,y_0)$}
						\put(4.2,34.5){{ $y_0$}}
						\put(19,7){$ t_0$}
						\put(7,7){$ 0$}
						\put(89,7){$ 2\pi$}
							\put(65,7){$ t_0+\alpha(y_0)$}
					\put(100,10){$ t$}
					\put(9,60){$ y $}
		\end{overpic}		
	\end{center}
		\caption{Intersections between the invariant tori in the unperturbed scenario and the subset $ \Sigma^{+} $. In the curve $ \To\times\{0\}\times\{y_0\} $, we can see the dynamics  generated by the impact map $ \mathcal{P}_0 $.}
	\label{fig3}
	\smallskip
\end{figure}
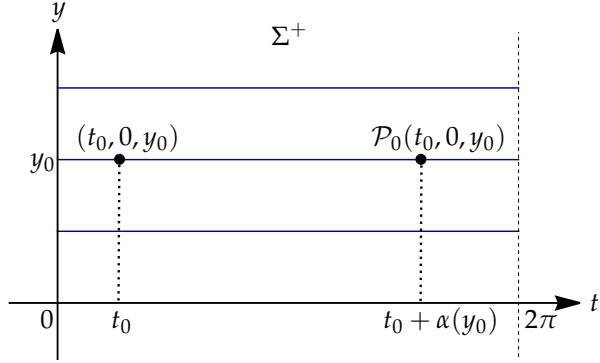

\begin{remark}
  As the map $ \mathcal{P}_{\e} $ is defined on $ \Sigma^{+} \subset \{x=0\} $, until the contrary is said, we will omit the $ x $-variable and treat $\mathcal{P}_{\e} $ as a function of $ (t_0, y_0) $, regarding it as a map on the annulus. The same convention will be applied to the maps $ \mathcal{P}_{\e} ^{+} $ and $ \mathcal{P}_{\e} ^{-} $.
\end{remark}
A fundamental concept for our approach is the exact-sympletic character for twist maps of the annulus. However, considering the cylinder $ \mathbb{T}\times \R^{+} $ endowed with the $ 2-$form $ \d t_0\wedge\d y_0 $, we cannot prove that the map $ \mathcal{P}_{\e} $ is symplectic. To overcome this obstacle, we re-write it in terms of a new couple of variables: $t_0$ and the associated symplectic conjugate, i.e., the energy $E_0$ according to the (piecewise) Hamiltonian structure of~\eqref{VF} (see~\eqref{H2}). The exact symplectic character of $ \mathcal{P}_{\e} $ in these variables follows at once from the application of Proposition~\ref{ExactSympleticGeneral} to this piecewise Hamiltonian. We postpone the proof of this proposition to Section~\ref{sympletic}.
\begin{proposition}\label{ExactSympleticGeneral}
  Let $ H(x,y,t)$ be a non-autonomous Hamiltonian, with respect to the $2$-form $\d x\wedge \d y$, and with $2\pi$-periodic dependence in $t$. By adding $E$ as a conjugate variable of $t$, we introduce the autonomous Hamiltonian ${\mathcal{H}}(x,t,y,E)=H(x,y,t)+E$, with respect to the $2$-form $\d x\wedge \d y+\d t\wedge \d E$. We denote the solutions associated with  $\mathcal{H}$ by
\[
  (x^{\tau}(x_0,t_0,y_0,E_0),t^{\tau}(x_0,t_0,y_0,E_0),y^{\tau}(x_0,t_0,y_0,E_0),E^{\tau}(x_0,t_0,y_0,E_0)),
  \]
  with $ t^{\tau}(x_0,t_0,y_0,E_0)=t_0 +\tau$, where $\tau$ represents the new time, and $(x_0,t_0,y_0,E_0)$ denotes the initial conditions at time  $\tau=0$.
  
  We consider the section $\Xi=\{x=0\}\cap\{\mathcal{H}=0\}$ and we suppose that the equation ${\mathcal{H}}(0,t,y,E)=0$ allows writing, at least locally for $(t,E)\in U\subset\R^2$, the variable $y$ as $y=y(t,E)$, being $y(t,E)$ a smooth function with $2\pi$-periodic dependence in $t$. We also suppose that it is possible to define the impact time $\tilde \tau(t_0,E_0)$ as the positive time for which we have
\begin{equation}\label{ImpTimeXi}
x^{\tilde{\tau}(t_0,E_0)}(0,t_0,y(t_0,E_0),E_0)=0,
\end{equation}
for any $(t_0,E_0)\in U$, being $\tilde \tau(t_0,E_0)$ a smooth function with $2\pi$-periodic dependence in $t_0$. Then, associated to this section $\Xi$, we introduce the map:
\[
\begin{aligned}
F:\hspace{0.3cm}U\hspace{0.35cm}&\longrightarrow\hspace{0.3cm}\R^2,\\
(t_0,E_0)&\longmapsto(t_1,E_1)=(t_0+f_{t_0}(t_0,E_0),f_{E_0}(t_0,E_0)),
\end{aligned}
\]
 where:
\[
	\begin{aligned}
		f_{t_0}(t_0,E_0)&=\tilde{\tau}(t_0,E_0),\\
		f_{E_0}(t_0,E_0)&=E^{\tilde\tau(t_0,E_0)}(0,t_0,y(t_0,E_0),E_0)=-H(0,y^{\tilde{\tau}(t_0,E_0)}(0,t_0,y(t_0,E_0),E_0),t_0+\tilde{\tau}(t_0,E_0)),
        \end{aligned}
\]
meaning that if we consider initial conditions $(0,t_0,y(t_0,E_0),E_0)\in\Xi$, then the corresponding solution integrated up to time $\tilde \tau(t_0,E_0)$, given by $(0,t_1,y_1,E_1)$, with $y_1=y^{\tilde{\tau}(t_0,E_0)}(0,t_0,y(t_0,E_0),E_0)$, also belongs to $\Xi$. Then, the map $F$ is exact symplectic with respect to the $1$-form $E_0\d t_0$, i.e., there is a function $S(t_0,E_0)$, $2\pi$-periodic in $t_0$, such that $F^\ast(E_1 \d t_1)=E_0\d t_0+\d S(t_0,E_0)$. 
\end{proposition}
We notice that the solutions of~\eqref{VF} are solutions of the non-autonomous Hamiltonian system for $ (x, y) $, with continuous piecewise analytic Hamilton function
	\begin{equation}\label{H1}
		\left\{\begin{array}{l l l}
			H^{r}_{\e} (x,y,t)=\dfrac{y^{2}}{2}+x(1-\e\; p(t))& \mbox{if} & x>0,\vspace{0.5cm}\\
			H^{l}_{\e} (x,y,t)=\dfrac{y^{2}}{2}-x(1+\e\; p(t))& \mbox{if} & x<0.
		\end{array}\right.
\end{equation}
Adding the energy $  E $ as a conjugate variable to the time $ t $, we transform $H^{\sigma}_{\e}$ into the two-degrees of freedom autonomous Hamiltonian
\begin{equation}\label{H2}
	\mathcal{H}^{\sigma}_{\e} (x,t,y,E)=H^{\sigma}_{\e} (x,y,t)+E,\quad \sigma=r,l,
\end{equation}
for which we denote as $\tau$ the new time variable. Notice that if we restrict to the zero energy level $ \mathcal{H}^{\sigma}_{\e}=0 $, then we recover all the solutions of $ H^{\sigma}_{\e}$, for $ \sigma=r,l $. If the solutions $t_{\sigma}^{\tau}$, $x_{\sigma}^{\tau}$, and $y_{\sigma}^{\tau}$ are known, then the expressions for $ E_{\sigma}^{\tau}$ are explicit from the relation $ \mathcal{H}^{\sigma}_{\e}=0 $. So, based on the construction of the half impact maps~\eqref{pmais} and~\eqref{pmenos}, it follows from Proposition~\ref{ExactSympleticGeneral} that both impact maps ${\mathcal{P}}^{+}_{\e} $ and ${\mathcal{P}}^{-}_{\e} $, in canonical coordinates $ (t_0,E_0) $, give rise to exact symplectic maps $ \bar{\mathcal{P}}^{+}_{\e} $ and $ \bar{\mathcal{P}}^{-}_{\e} $, respectively, with respect to the 1-form $E_0\d t_0$. Therefore their composition, to be denoted as $ \bar{\mathcal{P}}_{\e}=\bar{\mathcal{P}}^{-}_{\e} \circ \bar{\mathcal{P}}^{+}_{\e}$, is also exact symplectic.

From expressions~\eqref{H1} and~\eqref{H2}, it follows that both equations $\mathcal{H}^{\sigma}_{\e}=0$, for $ \sigma=r,l$, give rise, for $x=0$, to the relation $y^2/2+E=0$. So, once we have computed the map $(\bar t_0,\bar y_0)={\mathcal{P}}_{\e}(t_0,y_0)$, then the exact-symplectic map $(\bar t_0,\bar E_0)= \bar{\mathcal{P}}_{\e}(t_0,E_0)$ is obtained by replacing $y_0>0$ by $E_0=-y_0^2/2<0$ and $\bar y_0>0$ by $\bar E_0=-\bar y_0/2$.

To apply KAM theory to $ \bar{\mathcal{P}}_{\e} $ (in fact, to the localized and scaled map introduced in Proposition~\ref{P1}) we need, in particular, to handle the regularity of the map. In this paper, we regard into the analytical context as our framework for applying KAM theory. Although the differential equation~\eqref{eq1} is discontinuous, both $\bar{\mathcal{P}}_{\e}^+ $ and $ \bar{\mathcal{P}}_{\e}^- $ are analytic maps. This is because~\eqref{eq1} is piecewise analytic and both impact times $ \tmt (t_0,y_0;\e) $ and $ \tme (t_1,y_1;\e)$ are analytic functions of the initial conditions. Next point is to compute an integrable approximation for the maps $ {\mathcal{P}}_{\e}(t_0,y_0)$ and $ \bar{\mathcal{P}}_{\e}(t_0,E_0)$ and to control how far are them from this integrable approximation as $y_0\to +\infty$ and $E_0\to -\infty$, respectively. To achieve these purposes, we need to ensure that both $ \tmt (t_0,y_0;\e) $ and $ \tme (t_1,y_1;\e)$  are well defined for any $y_0\gg 0$ and $y_1\ll 0$, respectively, and to accurately analyze their behaviour for large values of $y_0$ and $y_1$. Since we are interested in dealing with analytical dependence, computations concerning $ \tmt$ and $\tme$ must be carried out not only for real values of the initial conditions, but also for values of $(t_0,y_0)$ and $(t_1,y_1)$ in appropriate complex strips around the real domains selected for each of these four variables. The width of these analyticity domains  can be selected independently of $\e$, provided that $\e>0$ is sufficiently small. We use the estimates on $ \tmt (t_0,y_0;\e) $ and $ \tme (t_1,y_1;\e)$ in complex domains to control ${\mathcal{P}}_{\e}^+$, ${\mathcal{P}}_{\e}^-$, and ${\mathcal{P}}_{\e}$. Since this is the most technical part of this work, we postpone the details to Section~\ref{propP}, and give here only the final result concerning the behaviour of the impact map ${\mathcal{P}}_{\e}$ in Proposition~\ref{P} below.

We introduce some notations to be used throughout the paper. First, we denote by $|z|$ the Euclidean norm of any complex number $z\in\C$ and we extend the same notation for the sup-norm of any complex valued vector or matrix. 
We also introduce:
\begin{equation}\label{Delta-norm}
	\Delta(\rho)=\{\theta\in\C\,:\, |{\rm Im}(\theta)|<\rho\},
	\quad
	\|f\|_{\rho}=\sup_{\theta\in\Delta(\rho)}\{|f(\theta)|\}.
\end{equation}
Here, $f=f(\theta)$ is a function (real valued, vector valued or complex valued), $2\pi$-periodic in $\theta$, that can be analytically extended to the complex strip $\Delta(\rho)$ of width $\rho>0$ and that it is bounded up to the boundary of this strip. Given positive quantities $y^\ast$, $\overline{\rho}$, and $\tilde{\rho}$, we define the following sets of $\C^2$:
\begin{eqnarray}
  \mathcal{D}^{+}(y^{*},\overline{ \rho},\tilde{ \rho}) & = & \{(t_0,y_0)\in\Delta(\overline{\rho})\times\D(\tilde{ \rho})\, :\, |y_0|>y^{*}\,\; \text{and} \,\; {\rm Re}(y_0)>0\},  \label{eq:calD+}
   \\
    \mathcal{D}^{-}(y^{*},\overline{ \rho},\tilde{ \rho}) & = & \{(t_1,y_1)\in\Delta(\overline{ \rho})\times\D(\tilde{ \rho})\, :\, |y_1|>y^{*}\,\; \text{and} \,\; {\rm Re}(y_1)<0\}.  \label{eq:calD-}
\end{eqnarray}
We will check that both ${\mathcal{P}}_{\e}^+ $ and ${\mathcal{P}}_{\e}$ are defined in sets of the form~\eqref{eq:calD+}, while $ {\mathcal{P}}_{\e}^- $ is defined in a set of the form~\eqref{eq:calD-}. 

We introduce the $2\pi$-periodic real analytic functions $\tilde P_j(t)$, for $j=-1,0,1,2$, defined as:
\begin{equation}\label{eq:p0p1p2}
\tilde P_0(t)=p(t)-a_0,\quad a_0=\ave{p}, \quad \tilde P_1'(t)=\tilde P_0(t),\quad \tilde P_2'(t)=\tilde P_1(t),\quad \tilde P_{-1}(t)=\tilde P_0'(t),
\end{equation}
where $ p(t)  $ is the real analytic $ 2\pi $-periodic function defined in~\eqref{eq:p} and $\ave{f}=\frac{1}{2\pi}\int_{0}^{2\pi}f(t) dt $ denotes the average of a $2\pi$-periodic function $f(t)$. The functions $\tilde P_1(t)$ and $\tilde P_2(t)$ are unique under the normalization $\ave{\tilde P_j}=0$, $j=1,2$. Due to the analyticity of $p(t)$, there are constants $0<\rho<1$ and $\tilde p>0$ such that
\begin{equation}\label{des:rho}
	 \|\tilde P_i\|_\rho\leq \tilde p, \quad \text{for} \quad j=-1,0 ,1,2.
\end{equation} 
 Both, $\rho$ and $\tilde p$, are set fixed throughout the paper.
\begin{proposition}\label{P}
With notations, definitions and hypotheses above, there are positive constants $\overline\rho$ and $\tilde\rho$, only depending on $\rho$, and positive numbers $0<\e^*_{\mathcal{P}}<\min\{1,\sfrac{1}{2|a_0|}\}$ and $\tilde C$, only depending on $a_0$, $\rho$, and $\tilde p$, such that the following holds. For any $0\leq\e<\e^*_{\mathcal{P}}$, the impact map $(\tbo,\ybo)={\mathcal{P}}_{\e}(t_0,y_0)$ is well defined and analytic if $(t_0,y_0)\in\mathcal{D}^{+}(4,\overline{\rho},\tilde{\rho})$ (see~\eqref{eq:calD+}), and takes the form
\begin{equation}\label{eq:Pe}
{\mathcal{P}}_{\e}:	\left\{\begin{aligned}
		\tbo&=t_0+ \alpha_{\e}(y_0)+\e f_{t_0}(t_0,y_0;\e),\\
		\ybo&=y_0+\e f_{y_0}(t_0,y_0;\e),
	\end{aligned}\right.
\end{equation}
with $ \alpha_{\e}(y_0)=\tfrac{4y_0}{1-a_0^{2}\e^{2}}$, being the functions $f_{t_0}(t_0,y_0;\e)$ and $f_{y_0}(t_0,y_0;\e)$ $2\pi$-periodic in $t_0$. Furthermore, the following estimates hold:
\begin{equation*}
		\left|\partialto^{i}\partialyo^{j}f_{t_0}(t_0,y_0;\e)\right| \leq \left\{\begin{aligned}
			&\tilde C    &\mbox{if} \quad j=0,\\
			& \frac{\tilde C}{|y_0|} & \mbox{if} \quad  j\neq 0,
		\end{aligned}\right.
	\qquad \text{and} \quad 
	   		\left|\partialto^{i}\partialyo^{j}f_{y_0}(t_0,y_0;\e)\right|\leq \tilde C,
	\end{equation*}
for any $ (t_0,y_0)\in\mathcal{D}^{+}(4,\overline{\rho},\tilde{\rho})$, $0\leq \e<\e^*_{\mathcal{P}}$, and $ 0\leq i+j\leq 2 $.
  \end{proposition} 
The impact map ${\mathcal{P}}_{\e}(t_0,y_0)$ lacks the desired exact symplectic character with respect to the 2-form $ \d t_0\wedge \d y_0 $ that we would like to have in order to be able to apply KAM theory to it. However, as noted above, this exactness in terms of the 2-form $ \d t_0\wedge \d E_0 $ is achieved by replacing the variable $y_0>0$ by $E_0=-y_0^2/2<0$. In this way, we obtain an analytic exact symplectic map $(\tbo,\Ebo)=\bar{\mathcal{P}}_{\e}(t_0,E_0)$, which takes the form: 
\begin{equation}\label{ImpactAngleEnergy}
\bar{\mathcal{P}}_{\e}:	\left\{\begin{aligned}
		\tbo&=t_0+\bar{\alpha}_{\e}(E_0) +\e \bar{f}_{t_0}(t_0,E_0;\e),\\
		\Ebo&=E_0+\e \bar{f}_{E_0}(t_0,E_0;\e),
	\end{aligned}\right.
\end{equation}
with $ \bar{\alpha}_{\e}(E_0)=\frac{4\sqrt{-2E_0}}{1-a_0^{2}\e^{2}} $, being the functions $\fbto(t_0,E_0;\e) $ and $ \fbEo(t_0,E_0;\e)$ $2\pi$-periodic in $t_0$. The complex domain of definition for $E_0$ of $\bar{\mathcal{P}}_{\e}$, as well as the corresponding estimates for $\bar{f}_{t_0}$ and $\bar{f}_{E_0}$, follow at once from the results of Proposition~\ref{P} on ${\mathcal{P}}_{\e}$, only taking into account the relations $E_0=-y_0^2/2<0$ and $\Ebo=-\ybo^2/2$. However, we are not going to state the analogue of Proposition~\ref{P} for $\bar{\mathcal{P}}_{\e}$ since, as noted above, our parameterization KAM theorem is not directly applied to $\bar{\mathcal{P}}_{\e}$ but to a map obtained through a process of localization and scaling of $\bar{\mathcal{P}}_{\e}$ around a specific action $E_0^*\ll 0$. This map is referred to as the scaled impact map, and its main properties are detailed in Proposition~\ref{P1}.

The map $\bar{\mathcal{P}}_{\e}$ takes the form of a close to integrable twist map of the annulus $\mathbb{T}\times\R$. Thus, the natural question arising from this fact is: which of the invariant curves for the unperturbed map (when $\e=0$) persist after small perturbations? This question could be answered by means of the celebrated KAM theory, like the Moser's twist mapping theorem~\cite{Moser1962} as well as other versions of this theorem (see e.g.~\cite{Levi1991, Liu2002,Ortega1999, Ortega2001}). However, all these results provide conditions for such persistence which, in particular, depend explicitly on the region considered and on the fact that some estimates uniformly hold in the selected region. Those reasons make very difficult the task to obtain persistence of invariant curves in unbounded domains, under conditions independent of the amplitude of the curve that we want to persist. Since the twist condition $\overline{\alpha}_{\e}'(E_0)$ for the impact map $\bar{\mathcal{P}}_{\e}$ tends to zero as $E_0\to-\infty$, we do not know any KAM result that can be globally applied in our context to show the persistence of curves of arbitrarily large amplitude.

In this work we address the persistence of these invariant curves in terms of a result based on the parametrization KAM theory introduced in~\cite{LGJV2005} (see~\cite{Haro2016} for a wide overview of parame\-tri\-za\-tion techniques in dynamical systems). Specifically, in \hyperref[TB]{Theorem B} we provide a KAM theorem concerning the persistence of a specific quasi-periodic invariant curve for a close to integrable, analytic and exact sympletic twist map $F$ of the form:
   \begin{equation}\label{TM}
     (\phi,I)\in\To\times\R\mapsto	F(\phi,I)=(\phi+\alpha(I)+f_{\phi}(\phi,I),I+f_I(\phi,I))\in\To\times\R,
   \end{equation}
 endowed by the $2-$form $\d\phi\wedge\d I=-\d(I\d\phi)$, i.e., verifying $F^\ast(I\d\phi)=I\d\phi+\d V(\phi,I)$, for some function $V(\phi,I)$ depending $2\pi$-periodically in $\phi$. By a specific invariant curve we mean the curve that is a perturbation of $\To\times\{I_0^{*}\}$, for some $I_0^{*}$, but that has the same frequency $\omega=\alpha(I_0^{*})$. Explicitly, given an analytic curve $\mathcal{T}\subset\To\times\R$, we say that it is $ F $-invariant with frequency $\omega $ if there is an analytic parametrization $\cur:\To\to\To\times\R$ of the curve $\mathcal{T}$ for which the following equation holds
\begin{equation}\label{invariance}
	F(\cur(\T))=\cur(\T+\omega),
\end{equation}
for all $ \T\in\R $. We refer to~\eqref{invariance} as the invariance equation for $\cur$. Then, if $\cur$ verifies~\eqref{invariance}, the curve $\mathcal{T}=\cur(\mathbb{T})$ is invariant by the map $F$ and the pull-back by $\cur$ of the dynamics on $\mathcal{T}$ becomes the rigid rotation on $\To$ of frequency $\omega$, i.e., $T_\omega(\cdot)=\cdot+\omega$. Hence, the rotation number of $\mathcal{T}$ is $\omega/2\pi$ and the dynamics on the curve is quasi-periodic if $\omega/2\pi$ is an irrational number. Actually, to discuss the persistence of this curve, we are going to assume that $\omega/2\pi$ is a Diophantine number of type $(\gamma,\nu)$ (see~\eqref{diophantine}). Since we are interested in invariant curves which are isotopic to $\To\times\{I_0^{*}\}$, we are going to consider parametrizations of the form $\cur(\theta)=(\theta+\cur_\phi(\theta),\cur_I(\theta))$, with $\cur_\phi$ and $\cur_I$ being $2\pi$-periodic in $\theta$.

The idea of the parametrization method is to solve for $\cur$ the equation~\eqref{invariance} by means of a quasi-Newton method. This method iteratively modifies $\cur$, but not the map $F$ which remains unchanged along the process. This fact eases the discussion of which conditions on $F$ are needed in order to ensure the persistence of the target curve with respect to the classical approaches to KAM theory which are based on the application of a sequence of canonical transformations to the map $F$. To measure the distance of $\cur$ from being a solution of~\eqref{invariance}, the invariance error associated to $\cur$ is defined as
   \begin{equation}\label{error}
	e(\T)=F(\cur(\T))-\cur(\T+\omega), \quad \T\in \To.
   \end{equation}
Hence, it is usually said that $\mathcal{T}=\cur(\To)$ is a quasi-torus of $F$ if the $2\pi$-periodic function $e$ is sufficiently small, in the sense of the norm $\|e\|_\rho$ for some $\rho>0$.
\begin{maintheoremB}\label{TB}
Consider $F=F(\phi,I)$ a real analytic and exact symplectic map of the annulus $\To\times\R$ of the form~\eqref{TM}, take a particular action $\Iot \in\R$ and denote $\omega=\alpha(\Iot)$. We suppose:
\begin{itemize}
\item[\textbf{(H1)}] 
$F$ can be analytically extended to $\mathcal{U}:= \D (\rho_0)\times D(I^{*}_0,R_0) $, for some $\rho_0,R_0\in(0,1)$, with $\D(\rho_0)$ being the complex strip introduced in~\eqref{Delta-norm} and $ D(\Iot,R_0) $ being the complex disc:
\begin{equation}\label{set:D}
	D(\Iot,R_0):=\{I\in\C\,:\,|I-I^{*}_0|<R_0\}.
\end{equation}
\item[\textbf{(H2)}] 
$\omega/2\pi$ is a Diophantine number of type $ (\gamma, \nu)$, for some $0<\gamma\leq 1$ and $\nu\geq2 $, i.e.,
\begin{equation}\label{diophantine}
	\left|\frac{\omega}{2\pi}-\frac{p}{q}\right|\geq \frac{\gamma}{q^{\nu}}\;, \quad \forall p\in\mathbb{Z}, \; \forall q \in\mathbb{N}.
\end{equation}
 \item[\textbf{(H3)}] 
 There are positive constants $c_1$, $c_2$, $c_3$ and $c$, with $c\leq 1$, for which the functions $ \alpha $, $ f_{\phi} $ and $ f_I $ in~\eqref{TM} satisfy:
\begin{align*}
		c_1\leq |\alpha'(I)|\leq c_2, \quad |\alpha''(I)|\leq c_3, \quad 	|\pphi^{i} \pI^{j} f_{\phi}(\phi,I)|\leq c, \quad	|\pphi^{i} \pI^{j} f_{I}(\phi,I)|\leq c,
\end{align*}	
for every  $ (\phi,I)\in\mathcal{U} $ and $ 0\leq i+j\leq 2 $. In particular, $\alpha'(\Iot)\neq 0$ and the (unperturbed) map $(\phi,I)\to (I+\alpha(I),I)$ is a twist map if $I\approx \Iot$.
\end{itemize}
Then, there is a constant $c^\ast\geq 1$ increasingly depending on $1/c_1$, $c_2$, $c_3$ and $\nu$, for which the following holds. Assume that $c$ is small enough so that it verifies
\begin{equation}\label{condTB}
	c^\ast\,c\leq\gamma^2\delta_0^{2\nu}\min\{c_1\delta_0/8,\delta_0/12,R_0/3\},\quad c^\ast\,c\leq 2^{-4\nu-2}\gamma^4\delta_0^{4\nu+1},
\end{equation}
where $\delta_0=\rho_0/12$. Then, there exists a real analytic function $\cur^\ast:\To\to\To\times\R$, of the form $\cur^\ast(\theta)=(\theta+\cur^\ast_\phi(\theta),\cur_I^\ast(\theta))$, with $\cur_\phi^\ast(\theta),\cur_I^\ast(\theta)$ being $2\pi$-periodic in $\theta$, and with $(\cur^\ast)'(\theta)\neq 0$ for all $\theta$, which turns out to be a solution of the invariance equation~\eqref{invariance}. Consequently, $\cur^\ast$ is the parametrization of an invariant curve $\mathcal{T}=\cur^\ast(\To)$ of $F$ with frequency $\omega$. Moreover, $\cur$ is defined for $\theta \in \D (\rho_0/2)$ and satisfies:
\begin{equation}\label{conclusionTB}
	\normindice{\cur_{\phi}^{*}}{\sfrac{\rho_0}{2}}\leq \frac{2\ce c}{\gamma^{2}\delta_0^{2\nu}}\;,
	\quad
	\normindice{\cur_{I}^{*}-I_0^{*}}{\sfrac{\rho_0}{2}}\leq \frac{2\ce c}{\gamma^{2}\delta_0^{2\nu}}\;,
	\quad 
	\normindice{(\cur_{s}^{*})'}{\sfrac{\rho_0}{2}}\leq \frac{2\ce c}{\gamma^{2}\delta_0^{2\nu+1}}\;,
\end{equation}
for $s\in\{\phi,I\}$.
\end{maintheoremB}
The proof of \hyperref[TB]{Theorem B} is postponed to Section~\ref{ParameterizationKam} which in turn is presented in a very synthetic way, since it is an adaptation to the context of previous results on parametrization KAM theory. The interested reader is referred to the references~\cite{LGJV2005,Haro2016} quoted above for more details and for the geometric motivation of the constructions used to prove it.
  \begin{remark}
In order to apply \hyperref[TB]{Theorem B}, we need $c$ (``the size of the perturbation'') to be smaller than an expression of $\mathcal{O}(\gamma^4)$, while classical KAM theory results usually only require $c$ to be smaller than an expression of $\mathcal{O}(\gamma^2)$. This discrepancy in the $\gamma$ exponent between classical and parametrization methods in KAM theory is due to the way in which the proof of the parametrization results is usually carried out. Getting a $\mathcal{O}(\gamma^2)$ condition for $c$ using KAM parametrization methods requires a more sophisticated approach than the one used here to prove \hyperref[TB]{Theorem B}. Such approach (see~\cite{Villanueva2018} for the details) is not considered here because a $\mathcal{O}(\gamma^4)$ condition for $c$ is sufficient for our purposes.
  \end{remark}
Using \hyperref[TB]{Theorem B} we are going to show that the persistence conditions of any unperturbed curve of the impact map $\bar{\mathcal{P}}_{\e}$ (of large enough amplitude) do not depend on the amplitude of this curve, but only on the Diophantine constant $\gamma$ of its frequency $\omega$. However, despite $ \bar{\mathcal{P}}_{\e} $ in~\eqref{ImpactAngleEnergy} is an analytic and quasi-integrable exact sympletic twist map of the annulus, some conditions of \hyperref[TB]{Theorem B} are not satisfied by it when $E_0\to -\infty$. We mainly stress the fact that the twist condition becomes extremely small, since $ \bar{\alpha}_{\e}'(E_0)=-\frac{4}{1-a_0^{2}\e^{2}}\frac{1}{\sqrt{-2E_0}}$. To address these shortcomings, we propose a construction in which \hyperref[TB]{Theorem B} is not applied directly to the map $\bar{\mathcal{P}}_{\e}$, but to a collection of maps that follow from the localization of $\bar{\mathcal{P}}_{\e}$ around a specific curve $ \mathbb{T}\times\{\Eot\}$ of the unperturbed approximation, followed by an appropriate scaling depending on the selected action $ \Eot $ (see Proposition~\ref{P1}). So, the the specific map to which we apply \hyperref[TB]{Theorem B} is different curve by curve, but the persistence conditions will be uniform in $\e$ for all $E_0^*\ll 0$
and, therefore, for all these maps. 

We consider fixed values of $\e>0$ and of the energy variable $\Eot<0$. We suppose that, for this couple, $\omega=\overline{\alpha}_{\e}(\Eot)$ is well defined (later on we will ask $\omega/2\pi$ to satisfy an appropriate Diophantine condition~\eqref{diophantine}). Then, we consider the change of variables (scaling) adapted to the selected value of $\Eot$:
\begin{equation}\label{eq:phiI}
\bar{\psi}:(\phi,I)\mapsto\left(t_0,E_0/\sqrt{-\Eot}\right).
\end{equation}
In particular, the new action $I$ corresponding to the selected energy variable $E=\Eot$ is $I=I^*_0=-\sqrt{-\Eot}$. The map $F(\phi,I)$ that we obtain after conjugation of $\bar{\mathcal{P}}_{\e}$ by $\bar{\psi}$ (we omit the explicit $\e$ dependence on $F$ since $\e$ is set fixed) is a real analytic exact symplectic map, that takes the form~\eqref{TM}, and that verifies the properties stated in Proposition~\ref{P1} below. For a better understanding of the statement of the Proposition~\ref{P1}, we point out that the relation between the variable $I<0$ introduced in~\eqref{eq:phiI} and the variable $y_0>0$ of Proposition~\ref{P} is 
\begin{equation}\label{eq:y0}
	y_0=y_0(I)=\sqrt{-\sqrt{2}\,y_0^*\,I},
\end{equation}
where $y_0^*=-\sqrt{2}\,I^*_0=\sqrt{-2\,\Eot}$ is the value of $y_0$ related to $E_0=\Eot$ and $I=I^*_0$.
 \begin{proposition}\label{P1}
  With the same notations and hypotheses of Proposition~\ref{P}, we consider a fixed couple $0<\e<\e^*_{\mathcal{P}}$ and $y_0^*>5$, and define $\Eot=-(y_0^*)^2/2$ and $I^*_0=-y_0^*/\sqrt{2}$. We then perform the scaling~\eqref{eq:phiI} to the map $\bar{\mathcal{P}}_{\e}$ of~\eqref{ImpactAngleEnergy} and we obtain a map $F(\phi,I)$ of the form
\begin{equation}\label{eq:FP1}
  F(\phi,I)=(\phi+\alpha(I)+f_{\phi}(\phi,I),I+f_I(\phi,I)),
  \end{equation}
with $\alpha(I)=\tfrac{4}{1-a_0^2\e^2}\sqrt{-\sqrt{2}\,y_0^*\,I}$ and  $f_{\phi}(\phi,I)$ and $f_I(\phi,I)$ being $2\pi$-periodic in $\phi$. The map $F$ is exact symplectic with respect to the 1-form $I\d\phi$ and it is real analytic if $(\phi,I)\in \Delta(\overline{\rho})\times{D}(I^*_0,\tilde\rho)$ (see~\eqref{Delta-norm} and~\eqref{set:D}), where $ \ov\rho $ and $ \tilde\rho $ are given in Proposition~\ref{P}. Moreover, there is a constant $\overline{C}$, that only depends on $a_0$, $\rho$, and $\tilde p$ (see~\eqref{eq:p0p1p2} and~\eqref{des:rho}), such that
  \[
  \frac{5\sqrt{2}}{3}<|\alpha'(I)|<\frac{10\sqrt{2}}{3}, \qquad
  |\alpha''(I)|\leq \frac{25}{24}, \qquad
  \left|\partial_\phi^{i}\partial_I^{j}f_{\phi}(\phi,I)\right| \leq \e\,\overline{C},\qquad \left|\partial_\phi^{i}\partial_I^{j}f_{I}(\phi,I)\right| \leq \e\,\overline{C},
  \]
for any $(\phi,I)\in \Delta(\overline{\rho})\times{D}(I^*_0,\tilde\rho)$, $0<\e<\e^*_{\mathcal{P}}$, and $ 0\leq i+j\leq 2 $.
\end{proposition}
The proof of Proposition~\ref{P1} is given in Section~\ref{propP1}.
It is worth noting that, if we vary the values of $I^\ast\to -\infty$, after the scaling we have that the size of the twist condition $\alpha'(I)$ has lower an upper bounds independent on $I^\ast$ in a disk of center $I^\ast$ an radius $\tilde\rho$ independent on $I^\ast$.
 
To conclude the proof of \hyperref[TA]{Theorem A}, we observe that the scaled impact map~\eqref{eq:FP1} satisfies the conditions of \hyperref[TB]{Theorem B}. Explicitly, we select a fixed $0 \leq \e < \e^*_{\mathcal{P}}$, where $ \e^*_{\mathcal{P}} $ is provided in Proposition~\ref{P}, and consider a value of $y=y_0^\ast>5$ for the impact map $\mathcal{P}_\e$ of the initial system~\eqref{eq1} (see Fig.~\ref{fig2}). According to the selected scaling, to such value $y=y_0^\ast$ it corresponds the action value $I=I_0^\ast=-y_0^\ast/\sqrt{2}$. Hence, for this given couple $(\e,y_0^\ast)$, the corresponding frequency of the unperturbed map $(\phi,I)\to (\phi+\alpha(I),I)$ in~\eqref{eq:FP1} is given by $\omega^\e_{y_0^\ast}=\alpha(I_0^\ast)=\frac{4}{1-a_0^2\e^2}y_0^\ast$. It is clear that, for a fixed $\e$, the expression $\omega^\e_{y_0^\ast}$ is a strictly increasing function of $y_0^\ast$ that goes to $+\infty$ as $y_0^\ast\to +\infty$. Let us assume that if $\omega=\omega^\e_{y_0^\ast}$, then $\omega/2\pi$ is a Diophantine number of type $ (\gamma, \nu)$ as defined in~\eqref{diophantine}. Then, we take $\rho_0=\bar\rho$, $R_0=\tilde\rho$, $c_1=\frac{5\sqrt{2}}{3}$, $c_2=\frac{10\sqrt{2}}{3}$, $c_3=\frac{25}{24}$ and $c=\e\,\overline{C}$, and apply \hyperref[TB]{Theorem B} to the map $F=F(\phi,I)$ in~\eqref{eq:FP1} corresponding to the selected couple $(\e,y_0^\ast)$ in terms of these parameters, at the action $I=I_0^\ast$. Consequently, all the conditions of \hyperref[TB]{Theorem B} are met and we conclude that there exists $0< \e^{*}\leq \e^*_{\mathcal{P}} $ such that $F$ possesses an analytic invariant curve of frequency $\omega=\omega^\e_{y_0^\ast}$, provided that the selected $\e$ is smaller than $\e^{*}$. Moreover, $\e^*$ does only depend on $\bar\rho$, $\tilde\rho$, $\overline{C}$, $\gamma$, and $\nu$, but not on the selected value of $y_0^\ast>5$ nor on the size of the frequency $\omega^\e_{y_0^\ast}$. This invariant curve of $F$ is close to the unperturbed circle $\To\times\{I_0^\ast\}$ and it admits an analytic parameterization $\cur^\ast=(\varphi^\e_{y_0^\ast})^\ast$, of the form $\cur^\ast(\theta)=(\theta+\cur^\ast_\phi(\theta),\cur_I^\ast(\theta))$, $\theta\in\To$, that solves the corresponding invariance equation~\eqref{invariance}. The parameterization $\cur^\ast$ also satisfies the estimates given in~\eqref{conclusionTB}, for some constant $c^\ast$ that only depends on $\nu$, where $\delta_0=\bar\rho/12$. In particular, this means that if we define $K^\ast=2\,c^\ast\,\overline{C}\gamma^{-2}\delta_0^{-2\nu-1}$, we have the bounds:
\[
  \normindice{\cur_{\phi}^{*}}{\sfrac{\bar\rho}{2}}\leq K^\ast\e,\qquad
  \normindice{\cur_{I}^{*}-I_0^\ast}{\sfrac{\bar\rho}{2}}\leq K^\ast\e, \qquad
  \normindice{(\cur_{s}^{*})'}{\sfrac{\bar\rho}{2}}\leq K^\ast\e,\quad s\in\{\phi,I\}. 
  \]
Going back to the original variables $ (t_0,x_0=0,y_0) $ of the impact map $\mathcal{P}_{\e}$, we observe that by considering the above discussed relations $t_0=\phi$ and $y_0=\sqrt{-\sqrt{2}\,y_0^*\,I}$ (see~\eqref{eq:phiI} and~\eqref{eq:y0}), we obtain the following analytical parametrization for the corresponding $\omega^\e_{y_0^\ast}$-quasi-periodic solution of $\mathcal{P}_{\e}$ in $\Sp$ (see~\eqref{SpSn}), that we denote as $\eta^+=(\eta^\e_{y_0^\ast})^+$ (in the following we will often omit the dependence on $(\e,y_0^\ast)$ since in this discussion we are dealing with a fixed pair of these values):
\[
  \eta^+(\theta)=(\theta+\eta_{t_0}^+(\theta),0,\eta_{y_0}^+(\theta))=\left(\theta+\cur^\ast_\phi(\theta),0,\sqrt{-\sqrt{2}\,y_0^*\,\cur_I^\ast(\theta)}\right),
  \quad
  \theta\in\To,
 \]
on the understanding that here we are treating $t_0$ as a variable defined modulus $2\pi$, so in fact we are treating the plane $\Sp$ as a cylinder, that we denote by $\tilde{\Sigma}^+$. This parametrization verifies $\mathcal{P}_{\e}(\eta^+(\theta))=\eta^+(\theta+\omega)$, where $\omega=\omega^\e_{y_0^\ast}$. Using that $y_0^\ast=-\sqrt{2}\,I_0^\ast$, we observe that
\begin{align*}
    \eta_{y_0}^+(\theta)-y_0^\ast &=\frac{-\sqrt{2}\,y_0^\ast\,(\cur_I^\ast(\theta)-I_0^\ast)}{\sqrt{-\sqrt{2}\,y_0^*\,\cur_I^\ast(\theta)}+y_0^\ast},
    \\
    (\eta_{y_0}^+)'(\theta) & =\frac{-\sqrt{2}\,y_0^\ast\,(\cur_I^\ast)'(\theta)}{\sqrt{-\sqrt{2}\,y_0^*\,\cur_I^\ast(\theta)}+y_0^\ast}+
 \frac{(y_0^\ast)^2\,(\cur_I^\ast(\theta)-I_0^\ast)\,(\cur_I^\ast)'(\theta)}{(\sqrt{-\sqrt{2}\,y_0^*\,\cur_I^\ast(\theta)}+y_0^\ast)^2\,\sqrt{-\sqrt{2}\,y_0^*\,\cur_I^\ast(\theta)}}.
\end{align*}
These expression allows to obtain the following estimates (we note that although $\eta^+$ is an analytic function, for stability purposes bounds for ``real'' values of $\theta$ suffice): 
\[
  |\eta_{t_0}^{+}(\theta)|\leq K^\ast\e,\quad
  |\eta_{y_0}^+(\theta)-y_0^\ast|\leq \sqrt{2}\,K^\ast\e,\quad
  |(\eta_{t_0}^{+})'(\theta)|\leq K^\ast\e,\quad
  |(\eta_{y_0}^+)'(\theta)|\leq 2\,\sqrt{2}\,K^\ast\e,
  \quad\theta\in\To.
  \]
To get the last bound of this list perhaps we have to make $\e^\ast$ slightly smaller or to slightly increase $K^\ast$, but in any case the obtained value of $\e^\ast$ or $K^\ast$ depends on the same parameters as the former one. Consequently, the closed curve $\mathcal{T}^+=(\mathcal{T}^\e_{y_0^\ast})^+=\eta^+(\T)\subset\tilde{\Sigma}^+$ is invariant by the action of the impact map $\mathcal{P}_{\e}$ and it remains close to the circle $\To\times\{y_0^\ast\}$, to which it is homotopic. For further uses, we also introduce the corresponding invariant curve in the cylinder $\tilde{\Sigma}^-$ defined by the action of $\mathcal{P}^+_\e$ on $\mathcal{T}^+$, i.e., $\mathcal{T}^-=(\mathcal{T}^\e_{y_0^\ast})^-=\mathcal{P}^+_\e(\mathcal{T}^+)\subset \tilde{\Sigma}^-$. Of course, we have $\mathcal{P}^-_\e(\mathcal{T}^-)=\mathcal{T}^+$.

By integration of $\mathcal{T}^+$ by the flow of~\eqref{VF}, this curve gives rise to an invariant two-dimensional torus $\mathcal{S}=\mathcal{S}^{\e}_{y_0^\ast}$ of~\eqref{VF}, on the understaning that we are still treating $t_0$ as a variable defined moduls $2\pi$. The torus $\mathcal{S}$ is continuous, picewise analytic and $\e$-close to the Cartesian product of the coresponding $y_0^\ast$-invariant curve of the unperturbed system $\e=0$ (as displayed in Fig.~\ref{fig1}) by the variable $t\in\To$, being homotopic to $(t,x,\dot x)\in\To\times\{0\}\times\{y_0\}$. Indeed, $\mathcal{S}$  is not smooth in both the upper curve $\mathcal{T}^+$ and the lower curve $\mathcal{T}^-$, since integration by~\eqref{VF} means to follow the flow $\Phi^{\tau}_r$ of~\eqref{solp} from $\mathcal{T}^+$ to $\mathcal{T}^-$ and the flow $\Phi^{\tau}_l$ of~\eqref{soln} from $\mathcal{T}^-$ to $\mathcal{T}^+$. Parametrizations of the right and left components of $\mathcal{S}$ are given by $\sigma^+=(\sigma^\e_{y_0^\ast})^+$ and $\sigma^-=(\sigma^\e_{y_0^\ast})^-$, respectively, which can be written as:
\[
\sigma^+(\theta,\tau)=\Phi^{\tau}_r(\eta^+(\theta)),\qquad \theta\in\To,\quad 0\leq\tau\leq\tau^+(\tilde{\eta}^+(\theta);\e),
\]
and
\[
\sigma^-(\theta,\tau)=\Phi^{\tau}_l(\eta^-(\theta)),\qquad \theta\in\To,\quad 0\leq\tau\leq\tau^-(\tilde{\eta}^-(\theta);\e),
\]
where $\tau^+$ and $\tau^-$ are the impact times introduced in~\eqref{eqaux1} and~\eqref{eqaux2}, and we are denoting $\eta^-=\mathcal{P}^+_\e(\eta^+)$ and $\tilde{\eta}^+$ and $\tilde{\eta}^-$ as the functions defined by the first and third components $(t,y)$  of the parameterizations $\eta^+$ and $\eta^-$, respectively (without the zero of the second component).

Since~\eqref{VF} is $2\pi$ periodic in $t$, to make the variable $t_0$ a true dynamical variable of the system we should lift it from $\To$ to $\R$. Then, the torus $\mathcal{S}=\mathcal{S}^{\e}_{y_0^\ast}$ becomes an infinite cylinder of $(t,x,\dot x)\in\R^3$, which is perpetually $\e$-close to the Cartesian product of one of the invariant curves of Fig.~\ref{fig1} by $t\in\R$. This means that it perpetually confines the trajectories whose initial conditions belongs inside this cylinder, so that they remain bounded for al time. Furthermore, the amplitude in the $(x,\dot x)$-plane of these cylinders goes to infinite as $y_0^\ast\to+\infty$.

The final part of the proof of \hyperref[TA]{Theorem A} is a mere observation. Let us consider $\omega_0\in\R$ any given value such that $\omega_0/2\pi$ is a Diophantine number according to definition~\eqref{diophantine}. E.g., we can select $\omega_0$ as $2\pi$ times the Golden Mean $(\sqrt{5}-1)/2$. Associated to $\omega_0$, we have a particular couple $ (\gamma, \nu)$ in~\eqref{diophantine}, that we set fixed up to the end of this section. We then introduce $\omega_k=\omega + 2\pi k$, where $k \in \mathbb{N}$. On the one side, $\omega_k\to +\infty$ as $k\to +\infty$. Consequently, since we have that $\omega^\e_{y_0^\ast}=\frac{4}{1-a_0^2\e^2}y_0^\ast$, then, for any given $0<\e<1/|a_0|$, we can define an infinite sequence of values of $y_0^\ast=y_0^\ast(\e,k)>5$ for which we have $\omega^\e_{y_0^\ast(\e,k)}=\omega_k$, for any $k\geq k_0$. The value of $k_0$ is independent of $\e$ and $y_0^\ast(\e,k)\to +\infty$ as $k\to +\infty$. On the other side, $\omega_k/2\pi$ is also a Diophantine number of type $ (\gamma, \nu)$, for any $k\in\mathbb{N}$. We then select as specific value of $\e^\ast$ in the statement of \hyperref[TA]{Theorem A} the one provided by construction above when applied to any given $y_0^\ast>5$ for which $\omega=\omega^\e_{y_0^\ast}$ is such that $\omega/2\pi$ is a Diophantine number of type $(\gamma, \nu)$. Since $\e^\ast$ depends on  $(\gamma, \nu)$, but not on the size of the frequency $\omega^\e_{y_0^\ast}$, we conclude that for each $0<\e<\e^\ast$ we can construct a sequence of two-dimensional tori $\{\mathcal{S}^{\e}_{y_0^\ast(\e,k)}\}_{k\geq k_0}$ whose projection onto the plane $(x,\dot x)$ confine the trajectories surrounded by each of these tori.  Since the amplitude of these tori go to infiniy with $k$, we get the desired perpetual stability for this specific value of $\e$.

\section{A parameterization result for invariant curves of a twist map of the annulus: Proof of Theorem~\ref{TB}}\label{ParameterizationKam}
Firstly, we introduce some notations and basic results to be used throughout the proof of \hyperref[TB]{Theorem B}. Let $g:\R\to\R$ be a real analytic and $2\pi$-periodic function that we assume that can be analytically extended to $\Delta(\rho)$ (see~\eqref{Delta-norm}), for some $\rho>0$, and denote by $\tilde{g}=g-\ave{g}$. Then, for any $0<\delta\leq\rho$ we have
\[
|\ave{g}|\leq\|g\|_\rho\;,\quad \|\tilde{g}\|_\rho\leq 2\|g\|_\rho\;,\quad \|g'\|_{\rho-\delta}\leq\frac{\|g\|_\rho}{\delta}.
\]
where the last inequality is a consequence of the application of the classical Cauchy estimates to $g$. Since $ \omega\in\R $ is set to be a fixed number, we introduce the notation $f_{+}(\cdot)=f(\cdot+\omega) $ and we define the operator $ \Lome(f)=f_+-f $. Let us denote as $g(\theta)=\sum_{k\in\Z}\hat{g}_k{\rm e}^{{\rm i} k\theta}$ the Fourier expansion of the $2\pi$-periodic and analytic function $g$ above. If $\omega/2\pi$ is an irrational number and $\ave{g}=\hat{g}_0=0$, then there is a unique formal $2\pi$-periodic solution $f=\Lome^{-1}(g)$ of the linear equation $\Lome(f)=g$, provided that the normalization condition $\ave{f}=0$ is assumed on $f$. This solution is explicitly given by the expansion:
\[
\Lome^{-1}(g)=\sum_{k\in\Z\setminus\{0\}}\frac{\hat{g}_k}{{\rm i} k\omega}{\rm e}^{{\rm i} k\theta}. 
\]
In fact, all the formal solutions of $\Lome(f)=g$ are of the form $f=c+\Lome^{-1}(g)$, for any constant $c$. If $ \omega/2\pi $ is a Diophantine number of type $ (\gamma,\nu) $ (see~\eqref{diophantine}), then so-called R\"ussmann's estimates (see e.g.~\cite{Haro2016}) show that $\Lome^{-1}(g)$ is also an analytic function and that there is a constant $C_R=C_R(\nu)$, for which $\|\Lome^{-1}(g)\|_{\rho-\delta}\leq C_R({\gamma\delta^{\nu}})^{-1}{\normindice{g}{\rho}}$. Finally, we recall that being $F$ an exact symplectic map of $\To\times\R $, then it is also a symplectic map and consequently the following equalities are satisfied
\begin{equation}\label{symplectic}
	F^{*}(\mbox{d}\phi \wedge \mbox{d}I)=\mbox{d}\phi \wedge \mbox{d}I \quad \mbox{and} \quad (DF(p))^{\top}J(DF(p))=J,
\end{equation}
where $ DF(p)$ denotes the differential matrix of $F$ at the point $ p$ and $J=\left(\begin{smallmatrix}0 & -1\\1 & 0\end{smallmatrix}\right)$ is the matrix representation of the two form $\mbox{d}\phi \wedge \mbox{d}I$.

  The proof of \hyperref[TB]{Theorem B} is performed by means of the iterative application of the lemma below. It is worthy mentioning that $ \cur $ and the related objects depend on $ \T $, but this dependence is going to be omitted in most of the expressions.
  \begin{lemmaiterativo}\label{iterativelemma}
	With the same hypotheses and notations on the statement of \hyperref[TB]{Theorem B}, let us consider a parametrization of a curve $\cur:\To\to\To\times\R$ of the form $\cur(\theta)=(\theta+\cur_\phi(\theta),\cur_I(\theta))$, for which the following holds. We suppose that $\cur_\phi$ and $\cur_I$ are $2\pi$-periodic functions that can be analytically extended to $\Delta(\rho)$ (see~\eqref{Delta-norm}), for some $0<\rho\leq \rho_0<1$, and verify:
\begin{equation}\label{h1}
		\normindice{\cur_{\phi}}{\rho}\leq \rho_0-\rho, \quad \normindice{\cur_{I}-I^{*}_0}{\rho}\leq R_0, \quad \normindice{\cur_{\phi}'}{\rho}\leq \sfrac{1}{4}\quad \mbox{and} \quad \normindice{\cur_{I}'}{\rho}\leq \sfrac{1}{4}.
\end{equation}
We denote by $e$ the invariance error~\eqref{error} associated to $\cur$ and $F$ and we define the scalar functions
\begin{equation}\label{OmegaA}
		\Omega=(\cur')^{\top}\cur',\qquad A=\Omega^{-1}\Omu (\tm')^{\top}DF(\cur)J\cur'.
\end{equation}
We assume that $|\ave{A}|\geq c_1/2$ and that there is a constant $0<\mu\leq 1$ for which $\|e\|_\rho\leq \mu$ and $\|e'\|_\rho\leq \mu$. Then, there is a constant $c^\ast\geq 1$ increasingly depending on $1/c_1$, $c_2$, $c_3$ and $\nu$, for which the following holds. We take $ 0<\delta<\sfrac{\rho}{3} $ for which
\begin{equation}\label{h3}
		\normindice{\cur_{I}-I_0^{*}}{\rho}\leq R_0-\frac{\ce\mu}{\gamma^{2}\de^{2\nu}},
		\qquad
		\normindice{\cur_{j}'}{\rho}\leq \frac{1}{4}-\frac{\ce\mu}{\gamma^{2}\de^{2\nu+1}},
\end{equation}
for $j\in\{\phi,I\}$. Then, we can define a new parametrization $\cur^{(1)}=\cur+\Delta\cur$, where $\Delta\cur:\To\to\To\times\R$ is $2\pi$ periodic and analytic in $\Delta(\rho^{(1)})$, with $\rho^{(1)}=\rho-3\delta$, which verifies:
\[
		\normindice{\Delta\cur}{\rho^{(1)}}\leq \frac{\ce\mu}{\gamma^{2}\de^{2\nu}}, \quad  \mbox{and} \quad  \normindice{(\Delta\cur)'}{\rho^{(1)}}\leq \frac{\ce\mu}{\gamma^{2}\de^{2\nu+1}}.
\]
In addition, if we define the invariance error $e^{(1)}$ and the function $A^{(1)}$ analogously as we have defined $e$ and $A$ for $\cur$, we also have:
\begin{equation*} 
		\normindice{A^{(1)}-A}{\rho^{(1)}}\leq \frac{\ce\mu}{\gamma^{2}\de^{2\nu+1}},  \quad  \normindice{e^{(1)}}{\rho^{(1)}}\leq  \mu^{(1)},
		\quad \mbox{and} \quad \normindice{(e^{(1))'}}{\rho^{(1)}}\leq  \mu^{(1)},
\end{equation*}
where $\mu^{(1)}= \frac{\ce\mu^{2}}{\gamma^{4}\de^{4\nu+1}} $.
\end{lemmaiterativo}
 Last formula means that the size of the new invariance error $e^{(1)}$ is (almost) quadratic with respect to the size of $e$. Hence, we may expect (under suitable assumptions) that the size of the error goes iteratively to zero very fast with the step.
 \begin{proof}
 Using Newton's method as a benchmark tool, a natural way to define the correction $\Delta\cur$ for the initial parametrization $\cur$ is to try to obtain a new error $e^{(1)}$ of quadratic size with respect to $e$. Explicitly:
\[
		e^{(1)}=F(\cur^{(1)})-\cur^{(1)}_+=E^{(1)}+\mathcal{R}(\cur,\D\cur),
\]
where 
\[
	E^{(1)}=e+DF(\cur)\Delta\cur-\Delta\cur_+,\quad\mathcal{R}(\cur,\D\cur)=F(\cur+\Delta\cur)-F(\cur)-DF(\cur)\Delta\cur.
\]
Hence, $ E^{(1)} $ stands for the linear approximation of $e^{(1)}$ around $\cur$ and $\mathcal{R}(\cur,\D\cur)$ is the corresponding quadratic reminder. In the aim of Newton's method, our ideal target is to try to set $E^{(1)}=0$. The functional structure of the problem makes this objective unrealistic, so we consider the less ambitious goal of obtaing an expression for $E^{(1)}$ of quadratic size in terms of the size of $e$. This is enough for our purposes since it allows to avoid the nocive effect that the small divisors associated to the problem have on the convergence of the iterative application of this lemma. A construction like the one used to define $\Delta\cur$ is usually called a quasi-Newton method. Since hypotheses on the statement of the lemma ensure that $\{\cur',J\cur'\}$ define a basis of $\R^2$ at any point of $\T$, we look for a correction $\D\cur$ of the form:
	\begin{equation}\label{delta}
		\D\cur=a\;\cur'+b\; \Omega^{-1} J \cur',
	\end{equation}
	where $a$ and $b$ are small scalar and $2\pi$-periodic functions which we select as the solutions of following (small-divisors) difference equations
	\begin{equation}\label{auxab}
		\mathcal{L}_{\omega}(b)=\ave{(\tm')^{\top}Je}-(\tm')^{\top}Je\quad \mbox{and} \quad		\mathcal{L}_{\omega}(a)=Ab+\Omu(\tm')^{\top}e,
	\end{equation}
	with the normalization condition $\ave{a}=0$, since selecting a difference value for $\ave{a}=0$ only means a ``change of origin'' on the curve. To avoid increasing the length of the proof, we omit details on the analytic and geometric motivations of selecting such specific expressions for $a$ and $b$, that are discussed in the references on parametric KAM Theory quoted in the paper. The right hand side of the equation for $b$ has zero average, so $b$ takes the form $b=\ave{b}+\tilde{b}$, where 
	\[
	\tilde{b}=\mathcal{L}_{\omega}^{-1}(\ave{(\tm')^{\top}Je}-(\tm')^{\top}Je),
	\]
	and $\ave{b}$ is selected in such a way the right hand side of the equation for $a$ has zero average. Using that by hypothesis $\ave{A}\neq 0$, this means that
	\[
	\ave{b}=-\ave{A}^{-1}\left(\ave{A\,\tilde{b}}+\ave{\Omu(\tm')^{\top}e}\right).
	\]
	Finally, we should select
	\[
	a=\mathcal{L}_{\omega}^{-1}(Ab+\Omu(\tm')^{\top}e).
	\]
	In order to establish that taking $\Delta\cur$ defined in such a way means a quadratic size for $E^{(1)}$ in terms of $e$, key computations are to show that the vectors $DF(\cur)\cur'$ and $DF(\cur)\Omega^{-1}J\cur' $ can be written in the following way
	\begin{equation}\label{C1}
		DF(\cur)\cur'=\cur_+' +e',\quad
		DF(\cur)\Omega^{-1}J\cur'=A\;\tm'+(1+B)\;\Omu J\tm',
	\end{equation}
where $A$ is defined in the statement and $B$ is the scalar function $B=-\Omega^{-1}(DF(\cur)^{-1}e')^{\top}\cur'$. Both expressions are straightforward. The first one follows by taking the derivative on the definition of the invariance error $e$ in~\eqref{error} and the second one follows by multiplying it by $(\cur')^\top$ and $(\cur')^\top\,J$, and using that $DF$ is a symplectic matrix (see~\eqref{symplectic}) at any point (we also recall that $J^\top=J^{-1}=-J$ and that $u^\top Ju=0$ for any vector $u$). The fact that $e$ is an analytic periodic function means that the size of $e'$ (and so the size of $B$) is ``comparable'' to the size of $e$ modulus the application of appropriate Cauchy estimates. In the literature, expressions in~\eqref{C1} are usually discussed in terms of the concept of quasi-reducibility of the the linearized system around a quasi-torus of a symplectic map. We refer the interested reader to the references quoted in the introduction for more details. Using expression in~\eqref{C1} and equation~\eqref{delta} it follows that:
	\[
	E^{(1)}=a\;e'+ B\;b\; \Omu J \tm'+\zeta^{(1)},
	\]
	where
	\[
	\zeta^{(1)}=e+(a-\am+Ab)\tm'+(b-\bm)\Omu J \tm'=
	e+(Ab-\mathcal{L}_{\omega}(a))\tm'-\mathcal{L}_{\omega}(b)\Omu J \tm'.
	\]
	By using difference equations~\eqref{auxab} for $a$ and $b$ and the fact that in the basis $ \{\tm',\Omu J\tm'\}$ the invariance error $ e $ is written as
	\[
	e=(\Omu(\tm')^{\top}e)\tm'-((\tm)^{\top}Je)\Omu J\tm',
	\]
	we get the expression $ \zeta^{(1)}=-\ave{(\tm')^{\top}Je}\Omu J \tm' $. Since we easily show that the size of $e'$, $B$, $a$, and $b$  are all comparable to the size of $e$, the quadratic behaviour in $e$ of $E^{(1)}$ is equivalent to the quadratic behaviour in $e$ of $ \zeta^{(1)}$ above. The fact that $ \zeta^{(1)}$ has quadratic size is not obvious, but it is a consequence of the exact symplectic properties of the map $F$. This is the only point of the proof in which the exactness of $F$ is considered. We remind that being $ F $ exact sympletic with respect to the $ 1$-form $ I\d \phi $ means that
	\[
	F^{*}( I\mbox{d}\phi)=I\mbox{d}\phi+\mbox{d}(V(\phi,I)),
	\]
	for some function $ V:\mathbb{T}\times \R \to \R $, $2\pi$-periodic in $ \phi $. By performing the pull-back by $\cur:\To\to\To\times\R$, we obtain:
	\[
	(F\circ\cur)^{*}( I\mbox{d}\phi))=\cur^{*}(I\mbox{d}\phi)+	\mbox{d}(V\circ\cur).
	\]
	By writting $\cur=\cur(\T)$ and $F=(F_\phi,F_I)$, and then considering the coordinate representation of this one-form in terms of the basis $\mbox{d}\theta$, we obtain the equality:
	\[
	(F_{I}(\cur))(F_{\phi}(\cur))'=(1+\cur_{\phi}')\cur_{I}+(V(\cur))'.
        \]
        Then, writting $e=(e_\phi,e_I)$ and using that $F_{\phi}(\cur)=\theta+\omega+\cur_{\phi,+}+e_{\phi}$, $F_{I}(\cur)=\cur_{I,+}+e_{I}$ and that $ (\tm')^{\top}Je=-e_{\phi}\cur_{I,+}'+e_{I}(1+\cur_{\phi,+}) $, it follows:
	\[
	(\tm')^{\top}Je= -\CL_{\omega}((1+\cur_{\phi}')\cur_{I})+(V(\cur))'-e_{\phi}'e_{I}-(e_{\phi}\cur_{I,+})'.
	\]
	Since $(1+\cur_{\phi}')\cur_{I}$, $e_{\phi}'e_{I}$, and $e_{\phi}\cur_{I,+}$ are $2\pi$-periodic in $\theta$, we conclude that $\ave{(\tm')^{\top}Je}=-\ave{e_{\phi}'e_{I}}$. Consequently:
	\[
	E^{(1)}=a\;e'+ B\;b\; \Omu J \tm'+\ave{e_{\phi}'e_{I}}\Omu J \tm'.
	\]
	This last computation ends the formal part of the proof of the lemma. Now, we have to perform the rigorous bounds of all the objects involved in above computations. In particular, bounds below show that the previous formal computations are well-defined if the hypotheses of the lemma hold. Our main purpose is to construct a constant $c^\ast$ for which the result holds. To achieve this aim, during the next sequence of bounds the value of $c^\ast$ is redefined recursively to meet a finite number of conditions. At the end of this process, the last value of $c^\ast$ obtained is that of the statement of the lemma.

        	From bounds on $\cur_\phi'$ and $\cur_I'$ in~\eqref{h1}, we notice that $\Omega=(1+\cur_\phi')^2+(\cur'_I)^2$ verifies $\|\Omega-1\|_\rho\leq 5/8$. Then, $3/8\leq |\Omega(\theta)|\leq 13/8$, for any $\theta\in\Delta(\rho)$ and so $\normindice{\Omega}{\rho}\leq 2$ and $\normindice{\Omega^{-1}}{\rho}\leq 3 $. Since hypotheses in~\eqref{h1} ensure that $DF(\cur(\theta))$ is well-defined for all $\theta\in\Delta(\rho)$, we conclude that $A$ is well defined in~\eqref{OmegaA} and verifies an estimate of the form $\|A\|_\rho\leq c^\ast$ for some $c^\ast$ that, up to now, only depends on $c_2$ (since we are assuming $c<1$).

                We then use R\"ussmann's estimates to control the size of the solutions for $a$ and $b$ of equations in~\eqref{auxab}. Firstly we have:
	\begin{equation}\label{tildeb}
		\normindice{\tilde{b}}{\rho-\delta}\leq c_R\frac{\normindice{\ave{(\tm')^{\top}Je}-(\tm')^{\top}Je}{\rho}}{\gamma\delta^{\nu}}\leq \frac{c^\ast\mu}{\gamma\delta^{\nu}},
	\end{equation}
	where $c_R=c_R(\nu)$ is the constant provided by R\"usmann's result, $\delta$ is the one in the statement of the lemma and $c^\ast$ is re-defined to meet both $\|A\|_\rho\leq c^\ast$ and~\eqref{tildeb}. By proceeding in analogous way as we have done to get these bounds, we obtain (by recursively re-defining $c^\ast$):
	\begin{align*}
		|\ave{b}| &\leq \frac{c^\ast\mu}{\gamma\delta^{\nu}},
		\quad
		\normindice{b}{\rho-\de}\leq  \frac{c^\ast\mu}{\gamma\delta^{\nu}},
		\quad
		\normindice{a}{\rho-2\de}\leq \frac{c^\ast\mu}{\gamma^{2}\delta^{2\nu}},
		\quad
		\normindice{\Delta\cur}{\rho-2\de}\leq \frac{c^\ast\mu}{\gamma^{2}\delta^{2\nu}},
		\\
		\normindice{(\Delta\cur)'}{\rho-3\de} & \leq \frac{c^\ast\mu}{\gamma^{2}\delta^{2\nu+1}},
		\quad
		\normindice{B}{\rho}\leq c^\ast\mu,
		\quad
		\normindice{E^{(1)}}{\rho-2\delta}\leq \frac{c^\ast\mu^2}{\gamma^{2}\delta^{2\nu}},
	\end{align*}
	where we are using standard Cauchy estimates to bound each $\theta$-derivatives and, to bound $B$, that the symplectic character of $F$ implies that $(DF)^{-1}=-J\,(DF)^\top\,J$.

	To control the remaining aspects of the proof, we need to ensure that the composition of $F$ and its derivatives with the new parametrization $\cur^{(1)}=\cur+\Delta\cur$ are well defined. By writting $\cur^{(1)}=(\theta+\cur_\phi^{(1)},\cur_I^{(1)})$ and $\Delta\cur=(\Delta\cur_\phi,\Delta\cur_I)$, and using hypotheses in~\eqref{h1} and~\eqref{h3}, we conclude that for any $\theta\in\Delta(\rho-2\delta)$ we have:
	\begin{align*}
		|\mbox{Im}(\cur_{\phi}^{(1)})| &\leq |\mbox{Im}(\T)|+\normindice{\cur_{\phi}}{\rho}+\normindice{\Delta\cur_{\phi}}{\rho-2\delta}\leq \rho_0-\delta< \rho_0,\\
		|\cur_I^{(1)}-I^{*}_0| &\leq \normindice{\cur_{I}-I^{*}_0}{\rho}+\normindice{\Delta\cur_{I}}{\rho-2\delta}\leq R_0.
	\end{align*}
	We note that although hypotheses in~\eqref{h3} are formulated in terms of the final value of $c^\ast$, they also valid in terms of the value of $c^\ast$ defined so far. In particular, bounds above allow us to define the new error $ e^{(1)} $. To obtain an estimate for the size of $ e^{(1)} $, it only remains to bound the remainder of the Taylor expansion $\CR(\cur,\Delta\cur)$, that we express as $\CR=(\CR_\phi,\CR_I)$. Then, to estimate $\CR_j(\cur,\Delta\cur)(\theta)$, for any fixed $\theta\in\Delta(\rho-2\delta)$ and with ${ j\in\{\phi,I\}} $, we introduce the assistant function $ g_j(s)=F_j(\cur+s\Delta\cur) $. Since $g_j'(s)=DF_j(\cur+s\Delta\cur)\,\Delta\cur$, integrating $ g_j $ by parts we observe that
	\[
	\CR_j(\cur,\Delta\cur)=g_j(1)-g_j(0)-g_j'(0)=-\int_{0}^{1}(s-1)g_j''(s)ds.
	\]
	Thus, $  \normindice{  \CR_j}{\rho-2\delta}\leq \sup_{s\in[0,1]}\{\normindice{g_j''(s)}{\rho-2\delta}\}$. But, since 
	\[
	g_j''(s)=\pphi^{2}F_{j}(\cur+ s\Delta\cur)(\Delta\cur_{\phi})^{2}+2\pI\pphi F_{j}(\cur+ s\Delta\cur)(\Delta\cur_{\phi})(\Delta\cur_{I})+\pI^{2}F_{j}(\cur+ s\Delta\cur)(\Delta\cur_{I})^{2},
	\]
	it follows that 
	\[
	\normindice{\CR}{\rho-2\delta} \leq \frac{\ce\mu^{2}}{\gamma^{4}\delta^{4\nu}},\qquad
	\normindice{e^{(1)}}{\rho-2\delta} \leq \frac{\ce\mu^{2}}{\gamma^{4}\delta^{4\nu}},\qquad
	\normindice{e^{(1)}}{\rho-3\delta} \leq \frac{\ce\mu^{2}}{\gamma^{4}\delta^{4\nu+1}},
	\]
	and therefore, the bounds for $ e^{(1)} $ and $ (e^{(1)}) '  $ are met in terms of the expression for $\rho^{(1)}$ and $\mu^{(1)}$ defined on the statement. To get estimates above on $g_j''(s)$, a crucial aspect is that $\cur+s\Delta\cur\in\Delta(\rho_0)$ for any $s\in [0,1]$, so that the composition of $\cur+s\Delta\cur$ with $F$ and its derivatives is well defined for any $s\in [0,1]$. It only remains to bound $A^{(1)}-A$, where
	\[
	A^{(1)} =\Omegau\Omuu ((\tm^{(1)} )')^{\top}DF(\cur^{(1)} ) J(\cur^{(1)} )'.
	\]
	We notice that hypotheses in~\eqref{h3} on $\cur_\phi'$ and $\cur_I'$ and bound above on $(\Delta\cur)'$ guarantee that $\Omega^{(1)}$ is well defined in $\Delta(\rho-3\delta)$ and verifies the same bounds as $\Omega$. To get the bound for $ A^{(1)}-A$, we consider the telescopic sum 
	\begin{align*}
		A^{(1)} -A  =&
		(\Omegau-\Omega^{-1})\Omuu ((\tm^{(1)} )')^{\top}DF(\cur^{(1)} ) J(\cur^{(1)} )'\\
		&+\Omega^{-1}(\Omuu-\Omega^{-1}_+)((\tm^{(1)} )')^{\top}DF(\cur^{(1)} ) J(\cur^{(1)} )'\\
		&+\Omega^{-1}\Omega^{-1}_+((\tm^{(1)} )'-\cur_+')^{\top}DF(\cur^{(1)} ) J(\cur^{(1)} )'\\
		&+\Omega^{-1}\Omega^{-1}_+(\cur_+')^{\top}(DF(\cur^{(1)} )-DF(\cur)) J(\cur^{(1)} )'\\
		&+\Omega^{-1}\Omega^{-1}_+(\cur_+')^{\top}DF(\cur)J(\cur^{(1)}-\cur)'.
	\end{align*}
	Then, we estimate every difference in the equality above leading to a bound of the form
	\[
	\normindice{A^{(1)}-A}{\rho-3\delta}\leq \frac{\ce\mu}{\gamma^{2}\de^{2\nu+1}}.
	\]
	The only expressions involved in this telescopic sum for which the corresponding bounds are not completely straightforward are $\Omegau-\Omega^{-1}$ and $DF(\cur^{(1)} )-DF(\cur)$. For the first one we write:
	\[
	\Omegau-\Omega^{-1}=\Omega^{-1}\Omegau(\Omega-\Omega^{(1)}),
	\]
	and we estimate the size of $\Omega^{(1)}-\Omega$ in terms of the size of $(\cur^{(1)})'-\cur'=(\Delta\cur)'$ by also expressing $\Omega^{(1)}-\Omega$ in terms of a telescopic sum analogous to the one above. Indeed, we get the estimate $\|\Omegau-\Omega^{-1}\|_{\rho-3\delta}\leq \frac{\ce\mu}{\gamma^{2}\de^{2\nu+1}}$. To bound $ DF(\cur^{(1)})-DF(\cur)$, we proceed analously as we have done to bound $\mathcal{R}(\cur,\D\cur)$. For any fixed $\theta\in\Delta(\rho-2\delta)$ and for any couple $j,k\in\{\phi,I\} $, we consider the auxiliary function $h_{j,k}(s)=\partial_j F_k(\cur+s\Delta\cur)$, for $s\in[0,1]$. By the Fundamental Theorem of Calculus, we see that
	\[
	\partial_j F_k(\cur^{(1)})-\partial_j F_k(\cur)=\int_{0}^{1}\nabla(\partial_j F_k)(\cur+s\Delta\cur)\, \d s\cdot\Delta\cur,
	\]
	where $\nabla$ refers to the gradient operator with respect to $(\phi,I)$. Since all the involved compositions are well-defined for any $\theta\in\Delta(\rho-2\delta)$ and $s\in[0,1]$, we get an estimate of the form $\normindice{\partial_j F_k(\cur^{(1)})-\partial_j F_k(\cur)}{\rho-2\delta}\leq  \frac{\ce\mu}{\gamma^{2}\de^{2\nu}}$ for any entry of $ DF(\cur^{(1)})-DF(\cur) $.
  \end{proof}

 \begin{proof}[Proof of Theorem~\ref{TB}]
Our goal is to prove this result by the iterative application of Lemma~\ref{iterativelemma}. Hence, we start by considering $\ce=\ce(1/c_1,c_2,c_3,\nu)\geq 1$ the value provided by this iterative lemma, that we set fixed from now on along the proof of the theorem.
	\\
	\noindent{\bf First step of the iterative scheme.} To perform the first application (for $n=0$) of Lemma~\ref{iterativelemma}, we select the parametrization $\cur(\theta)=\tauind{0}(\T)=(\T,I^{*}_0)$, so that $\cur_\phi(\T)=0$ and $\cur_I(\theta)=I^{*}_0$. We also consider the constant values $\rho=\rho^{(0)}=\rho_0$ and $\delta=\delta^{(0)}=\rho_0/12$. In particular, $\rho^{(1)}=\rho^{(0)}-3\delta^{(0)}>0$. Clearly, the parametrization $\cur$ is analytic in $\Delta(\rho)=\Delta(\rho_0)$, and we can easily check that:
	\[
	\eind{0}(\T)=(f_{\phi}(\T,I_0^{*}),f_I(\T,I_0^{*})),
	\qquad
	\Aind{0}(\T)=-\alpha'(I_0^{*})-\pI f_{\phi}(\T,I_0^{*}).
	\]
	Hence, we can set $\mu=\mu^{(0)}=c$ and we have:
	\begin{align*}
		\normindice{\tauind{0}_{\phi}}{\rhoind{0}} & =0=\rho_0-\rhoind{0},
		\quad
		\normindice{\tauind{0}_I-I_0^{*}}{\rhoind{0}}=0\leq R_0-\frac{\ce\mu^{(0)}}{\gamma^{2}(\de^{(0)})^{2\nu}},\\
		\normindice{(\tauind{0}_s)'}{\rhoind{0}} & =0\leq\frac{1}{4}-\frac{\ce\mu^{(0)}}{\gamma^{2}(\de^{(0)})^{2\nu+1}},
		\quad
		|\ave{\Aind{0}}|\geq c_1-c\geq \frac{c_1}{2},
	\end{align*}
	for $s\in\{\phi,I\}$, provided that $\frac{\ce\mu^{(0)}}{\gamma^{2}(\de^{(0)})^{2\nu}}\leq\min\{R_0,\delta^{(0)}/4\}$ and $c\leq c_1/2$, conditions that are guaranteed by hypotheses~\eqref{condTB} on the statement. Thus, we can apply Lemma~\ref{iterativelemma} to $\cur^{(0)}$ and we can define a new parametrization of quasi-torus $\tauind{1}$ that verifies bounds on the statement of the lemma.\\
	\noindent{\bf A close expression for the iterative error.}
	To iterativelly apply Lemma~\ref{iterativelemma}, we are going to select, at each step $n\geq 0$ of the iterative scheme, the constant values $\delta=\delta^{(n)}=\delta^{(0)}/2^n$ and $\rho=\rho^{(n)}=\rho^{(0)}-3\sum_{j=0}^{n-1}\delta^{(j)}$, for which we clearly have $\rho^{(n)}\geq\rho_0/2$, for any $n\geq 0$. Let us suppose for a moment that we can freely iterate for any $n\geq 0$ by dealing with these selected values. Then, we have $\mu^{(0)}=c$ and we get the following iterative expression for the error $\mu=\mu^{(n)}$ after $n$ iterations:
	\[
		\muind{n+1}=\frac{\ce (\muind{n})^{2}}{\gamma^{4}(\deltaind{n})^{4\nu+1}}=\frac{\ce}{\gamma^{4}(\deltaind{0})^{4\nu+1}}2^{(4\nu+1)n} (\muind{n})^{2}.
	\]
	Simple computations show that we obtain
	\begin{align}
		\muind{n}&=\left(\frac{\ce}{\gamma^{4}(\deltaind{0})^{4\nu+1}}\right)^{2^{n}-1}2^{(4\nu+1)(2^{n}-(n+1))}(\muind{0})^{2^{n}}\nonumber\\
		&=\left(\frac{\ce}{\gamma^{4}(\deltaind{0})^{4\nu+1}}\right)^{-1}2^{-(4\nu+1)(n+1)}\left(\frac{2^{4\nu+1} \ce \muind{0}}{\gamma^{4}(\deltaind{0})^{4\nu+1}}\right)^{2^{n}},\label{mun}
	\end{align}
	expression that can easily be verified by induction. In particular, since hypotheses on $\mu^{(0)}=c$ mean that $\frac{2^{4\nu+1} \ce \muind{0}}{\gamma^{4}(\deltaind{0})^{4\nu+1}}\leq \frac{1}{2}$, we conclude that this expression verifies $\muind{n}\to 0$ as $n\to\infty$. In addition, we also have that the following estimates hold for every $n\geq 0$:
	\[
	\frac{\ce\muind{n}}{\gamma^{2}(\deltaind{n})^{2\nu}}\leq \frac{\ce \muind{0}}{\gamma^{2}(\deltaind{0})^{2\nu}}\left(\frac{1}{2}\right)^n,
	\qquad
	\frac{\ce\muind{n}}{\gamma^{2}(\deltaind{n})^{2\nu+1}}\leq \frac{\ce \muind{0}}{\gamma^{2}(\deltaind{0})^{2\nu+1}}\left(\frac{1}{2}\right)^n.
	\]
	Indeed, using that $2^n\geq n+1$ and that $\frac{2^{4\nu+1} \ce \muind{0}}{\gamma^{4}(\deltaind{0})^{4\nu+1}}\leq \frac{1}{2}$, we have:
	\begin{align*}
		\frac{\ce\muind{n}}{\gamma^{2}(\deltaind{n})^{2\nu}}&=\frac{\gamma^{2}(\deltaind{0})^{2\nu+1}}{2^{2\nu}} 2^{-(2\nu+1)(n+1)}\left(\frac{2^{4\nu+1}\ce \muind{0}}{\gamma^{4}( \deltaind{0})^{4\nu+1}}\right)^{n+1}\\
		&=\frac{2^{2\nu+1}\ce \muind{0}}{\gamma^{2}(\deltaind{0})^{2\nu}}2^{-(2\nu+1)(n+1)}\left(\frac{1}{2}\right)^n\\
		&\leq \frac{\ce \muind{0}}{\gamma^{2}(\deltaind{0})^{2\nu}}\left(\frac{1}{2}\right)^n,
	\end{align*}
	and analogously for the other case.\\
	{\bf General step.} We iteratively apply Lemma~\ref{iterativelemma} starting with the $0$-data selected above and suppose that we have been able to iterate $n$ times, for some $n\geq 1$, so we define the family of corrected parameterizations $ \{\tauind{j}\}_{j=0}^{n} $, for which the estimates provided by Lemma~\ref{iterativelemma} hold in the corresponding domain $\Delta(\rho^{(j)})$, for each $j=1,\ldots,n$. In particular, the size of the errors $\|e^{(j)}\|_{\rho^{(j)}}$ and $\|(e^{(j)})'\|_{\rho^{(j)}}$ is controlled by $\mu^{(j)}$ given by expression in~\eqref{mun}. We want to show that we can iterate again by showing that $\tauind{n}$ fit into the conditions needed to apply the Iterative Lemma. Specifically, we need to check, for any $n\geq 1$:
	\[
	\normindice{\cur_{\phi}^{(n)}}{\rho^{(n)}}\leq \rho_0-\rho^{(n)},
	\quad
	\normindice{\cur_{I}^{(n)}-I_0^{*}}{\rho^{(n)}}\leq R_0-\frac{\ce\mu^{(n)}}{\gamma^{2}(\de^{(n)})^{2\nu}},
	\quad
	\normindice{(\cur_{s}^{(n)})'}{\rho^{(n)}}\leq \frac{1}{4}-\frac{\ce\mu^{(n)}}{\gamma^{2}(\de^{(n)})^{2\nu+1}},
	\]
	for $s\in\{\phi,I\}$, and $|\ave{A^{(n)}}|\geq c_1/2$. Notice that
	\[
	\normindice{\tauphiind{n}}{\rhoind{n}}
	\leq
	\sum_{j=0}^{n-1} \normindice{\Delta\tauphiind{j }}{\rhoind{j+1}}
	\leq
	\sum_{j=0}^{n-1}\frac{\ce \muind{j}}{\gamma^{2}(\deltaind{j})^{2\nu}}
	\leq
	\frac{\ce \muind{0}}{\gamma^{2}(\deltaind{0})^{2\nu}}\sum_{j=0}^{n-1}\left(\frac{1}{2}\right)^j
	\leq
	\frac{2\ce \muind{0}}{\gamma^{2}(\deltaind{0})^{2\nu}},
	\]
	and the same bound holds for $\normindice{\cur_{I}^{(n)}-I_0^{*}}{\rho^{(n)}}$. We proceed analogously to get the bounds $\normindice{(\cur_{s}^{(n)})'}{\rho^{(n)}}\leq \frac{2\ce \muind{0}}{\gamma^{2}(\deltaind{0})^{2\nu+1}}$, for $s\in\{\phi,I\}$, by means of the sum of the bounds on $\normindice{(\Delta\cur_{s}^{(n)})'}{\rho^{(j+1)}}$ for $j=0,\ldots,n-1$. Finally, to control the average $\ave{A^{(n)}}$ we start by bounding 
	\[
	\normindice{\Aind{n}-\Aind{0}}{\rhoind{n}}\leq \sum_{j=0}^{n-1} \normindice{\Aind{j+1}-\Aind{j}}{\rhoind{j+1}}\leq \frac{2\ce \muind{0}}{\gamma^{2}(\deltaind{0})^{2\nu+1}}.
	\]
	Then, bounding the size of the average $\ave{A^{(n)}-A^{(0)}}$ by the norm of the function and using the estimate above $|\ave{\Aind{0}}|\geq c_1-c$, we obtain:
	\[
	|\ave{A^{(n)}}|\geq |\ave{\Aind{0}}|-|\ave{A^{(n)}-A^{(0)}}|\geq |\ave{\Aind{0}}|-\normindice{\Aind{n}-\Aind{0}}{\rhoind{n}}\geq c_1-c-\frac{2\ce \muind{0}}{\gamma^{2}(\deltaind{0})^{2\nu+1}}.
	\]
	Using that $\rho_0-\rho^{(n)}\geq\rho_0-\rho^{(1)}=3\delta^{(0)}=\rho_0/4$, for any $n\geq 1$, the bound $\normindice{\cur_{\phi}^{(n)}}{\rho^{(n)}}\leq \rho_0-\rho^{(n)}$ is achieved provided that $\frac{2\ce \muind{0}}{\gamma^{2}(\deltaind{0})^{2\nu}}\leq \rho_0/4$. Next, using that $\frac{\ce\muind{n}}{\gamma^{2}(\deltaind{n})^{2\nu}}\leq\frac{\ce \muind{0}}{\gamma^{2}(\deltaind{0})^{2\nu}}$ and $ \frac{\ce\muind{n}}{\gamma^{2}(\deltaind{n})^{2\nu+1}}\leq \frac{\ce \muind{0}}{\gamma^{2}(\deltaind{0})^{2\nu+1}}$, the desired bounds on $\normindice{\cur_{I}^{(n)}-I_0^{*}}{\rho^{(n)}}$, $\normindice{(\cur_{s}^{(n)})'}{\rho^{(n)}}$ and $|\ave{A^{(n)}}|$, are fullfiled provided that $\frac{3\ce \muind{0}}{\gamma^{2}(\deltaind{0})^{2\nu}}\leq R_0$, $\frac{3\ce \muind{0}}{\gamma^{2}(\deltaind{0})^{2\nu+1}}\leq 1/4$ and $c+\frac{2\ce \muind{0}}{\gamma^{2}(\deltaind{0})^{2\nu+1}}\leq c_1/2$, respectively. All these bounds are guaranteed by the conditions on $c$ on the statement of the theorem. \\
	{\bf Convergence of the iterative scheme.} Once we have verified that hypotheses on the statement allow us to compute the full sequence of iterates $\{\cur^{(n)}\}_{n=0}^\infty$, the convergence of $\cur^*=\lim_{n\to\infty}\cur^{(n)}$ in $\Delta(\rho^*)$, with $\rho^*=\lim_{n\to\infty}\rho^{(n)}=\rho_0/2$, is straightforward as the series $\cur^*=\cur^{(0)}+\sum_{n=0}^\infty\Delta\cur^{(n)}$ is absolutely convergent in terms of the norm $\|\cdot\|_{\rho_0/2}$. Estimates on $\cur^*$ on the statement are also straightforward from computations above.
\end{proof}

\section{The Impact Map: Proof of Proposition~\ref{P}}\label{propP}
From the $2\pi$-periodic function $p$ of the statement, we define the functions
\[
	P_1(\tau,t_0):=\int_{0}^{\tau}p(s+t_0)\d s \quad \mbox{and} \quad P_2(\tau,t_0):=\int_{0}^{\tau}P_1(s,t_0)\d s.
\]
By considering the function $ \Put $, $ \Pdt $ and $a_0$  introduced in~\eqref{eq:p0p1p2}, we can rewrite $ P_1 $ and $ P_2 $ as
\begin{equation}\label{eq:P1P2}
	P_1(\tau,t_0)=a_0\tau+\tilde{P}_1(t_0+\tau)-\tilde{P}_1(t_0) \quad \mbox{and} \quad P_2(\tau,t_0)=a_0\frac{\tau^{2}}{2}-\tau \tilde{P}_1(t_0) +\tilde{P}_2(t_0+\tau)-\tilde{P}_2(t_0).
\end{equation}

With the functions $ P_1 $ and $ P_2 $ defined in~\eqref{eq:P1P2} we can provide the explicit solutions of the system originating from~\eqref{VF}. For $ x_0>0 $, we have
\begin{equation}\label{eq:x>0}
	\left\{\begin{aligned}
		&t_r^{\tau}(t_0,x_0,y_0;\e)=\tau+t_0,\\
	&x_r^{\tau} (t_0,x_0,y_0;\e)= x_0 +\tau y_0-\frac{{\tau}^{2}}{2}  + \e\;P_2(\tau,t_0),\\
	&y_r^{\tau}(t_0,x_0,y_0;\e)=  y_0 -\tau +\e\; P_1(\tau,t_0).
	\end{aligned}\right.
\end{equation}\vspace{0.2cm}
If $ x_1<0 $, we have
\begin{equation}\label{eq:x<0}
	\left\{\begin{aligned}
		&t_l^{\tau}(t_1,x_1,y_1;\e)=\tau+t_1,\\
		&x_l^{\tau}(t_1,x_1,y_1;\e)= x_1 +\tau y_1+\frac{{\tau}^{2}}{2}  + \e\;P_2(\tau,t_1), \\
		&y_l^{\tau}(t_1,x_1,y_1;\e)=  y_1 +\tau +\e\;P_1(\tau,t_1).
	\end{aligned}\right.
\end{equation}
By expressions in~\eqref{eq:x>0} and~\eqref{eq:x<0}, we notice that the times of impact $ \tmt (t_0,y_0;\e) $ and $  \tme (t_1,y_1;\e)$, introduced in~\eqref{eqaux1} and~\eqref{eqaux2}, respectively, cannot be explicitly obtained, so we cannot deal with explicit expressions for the impact maps $ \mathcal{P}_{\e}^{+}$, $ \mathcal{P}_{\e}^{-}$, and $ \mathcal{P}_{\e}$. For this reason, we propose the following approach to deal with them.

\noindent{\bf Asymptotic behaviour of the first impact time.}
Since we are interested in discussing the persistence of quasi-periodic invariant curves of $ \mathcal{P}_{\e}$ of arbitrarily large amplitude $ y_0 $ and $0<\e\ll 1$, we start by discussing an asymptotic expression for $ \tmt (t_0,y_0;\e) $ and $  \tme (t_1,y_1;\e)$ when $y_0\to +\infty$ and $y_1\to -\infty$, respectively. Using expressions~\eqref{eq:P1P2} for $P_2$ and~\eqref{eq:x>0} for $ x_r^{\tau}  $, equation~\eqref{eqaux1} for the first impact time $\tau^{+}=\tau^{+}(t_0,y_0;\e)$ reads as:
		\begin{equation}\label{eq:eqtau+}
		\tau^{+}y_0-\dfrac{1-a_0\e}{2}(\tau^{+})^2-\e\tau^{+}\tilde{P}_1(t_0)+\e\tilde{P}_2(t_0+\tau^{+})-\e\tilde{P}_2(t_0)=0.
		\end{equation}
                Hence, for $ y_0\gg 0$ sufficiently large, one expects the dominant part of the expression of $ \tau^{+} $ to be:
               \begin{equation}\label{eq:tau+0}
			\tau^{+}_0(t_0,y_0;\e)=\frac{2(y_0-\e\tilde{P}_1(t_0))}{1-a_0\e},
	       \end{equation}
               so it is natural to write
              	\begin{equation}\label{taumais}
		  \tau^{+}(t_0,y_0;\e)=\tau^{+}_0(t_0,y_0;\e)+\e \tau_*^{+}(t_0,y_0;\e),
	        \end{equation}
	        where $\tau_*^{+} $ is a small correction
(as $y_0\to +\infty$) satisfying properties to be given in the next lemma.
                \begin{lemma} \label{lemataup}
                  Let $\tilde{P}_1$ and $ \tilde{P}_2$ be the real analytic and $2\pi$-periodic functions introduced in~\eqref{eq:p0p1p2} which verify $ \tilde{P}_2'=\tilde{P}_1$  and $\|\tilde P_j\|_\rho\leq \tilde p$, $j=1,2$, for $0<\rho<1$ and $\tilde p>0$. Given $ y_0^{*} \geq 1$, $0<\overline{\rho}_0\leq\rho$, $\overline{ \rho}_0^{+}>0 $, and $ \tilde{\rho}_0^{+} >0$, such that $\overline{\rho}_0^{+}+4  \tilde{\rho}_0^{+} <\overline{\rho}_0$, we consider the set $\mathcal{D}^{+}_0=\mathcal{D}^{+}(y_0^{*},\overline{ \rho}_0^{+},\tilde{ \rho}_0^{+})$ (see~\eqref{eq:calD+}) and define $\e_*^+$ as
           \[
	     \e^{+}_{*}=\sup\left\{\e>0:  |a_0\e|\leq \dfrac{1}{2}\quad\mbox{and}\quad \overline{\rho}_0^{+}+4\tilde{\rho}_0^{+}+288\e\tilde{p}<\overline{\rho}_0 \right\}.
           \]
           Then, for every $ 0<\e\leq  \e^{+}_{*}$, the first impact time $\tau^{+}$ can be written as in~\eqref{eq:tau+0} and~\eqref{taumais}, with $  \tau_0^{+}(t_0,y_0;\e) $ and $  \tau_*^{+}(t_0,y_0;\e) $ being real analytic functions for $(t_0,y_0)\in\mathcal{D}^{+}_0$, $2\pi$-periodic in $t_0$, and satisfying:
	\begin{equation}\label{taumaiszeroandast}
	  \frac{|y_0|}{2}\leq |\tau^{+}_0(t_0,y_0;\e)|\leq 6|y_0|,
          \qquad |\tau_*^{+}(t_0,y_0;\e)|\leq \dfrac{32\tilde{p}}{|y_0|}.
	\end{equation}
 Moreover, given any $0<\hat\rho_0^+<\min\{\overline{\rho}_0^{+},\tilde{ \rho}_0^{+}\}/2$, there exists $ C_*^{+}= C_*^{+}(\tilde p,\hat\rho_0^+)>0$ such that 
		\begin{equation}\label{eq:C*+}
			\left|\partial^{i}_{t_0}\partial^{j}_{y_0}\tau_*^{+}(t_0,y_0;\e)\right|\leq \frac{C_*^{+}}{|y_0|},
		\end{equation}
		for every $ (t_0,y_0)\in\mathcal{D}^{+}(y_0^{*}+2\hat\rho_0^+,\overline{ \rho}_0^{+}-2\hat\rho_0^+,\tilde{ \rho}_0^{+}-2\hat\rho_0^+)$, $0<\e\leq\e^{+}_{*}$, and $ 0\leq i+j\leq 2 $.
\end{lemma}
                \begin{proof}
We start by observing that for any $(t_0,y_0)\in \mathcal{D}^{+}_0$ and $0<\e\leq \e^{+}_*$, the function $\tau^{+}_0$ of~\eqref{eq:tau+0} satisfies the estimates in~\eqref{taumaiszeroandast}. This is straightforward by observing that $|y_0|\geq y_0^*\geq 1$, $|a_0\e|\leq \dfrac{1}{2}$, and that $\e\tilde{p}<\dfrac{1}{2}$. Indeed, using that $1/2\leq |1-a_0\e|\leq 3/2$, we have:
		\[
			\frac{1}{2}\leq \frac{2}{3}\leq \frac{2}{|1-a_0\e|}\left(1-\dfrac{\e|\tilde{P}_1(t_0)|}{|y_0|}\right) \leq \left|\frac{\tau^{+}_0(t_0,y_0;\e)}{y_0}\right| \leq \frac{2}{|1-a_0\e|}\left(1+\dfrac{\e|\tilde{P}_1(t_0)|}{|y_0|}\right)\leq 6.
		\]
	        If we write $\tau^{+}$ as in~\eqref{taumais}, and we impose that it satisfies equation~\eqref{eq:eqtau+}, we have that $\tau_*^{+}=\tau_*^{+}(t_0,y_0;\e)$ satisfies the equation $\tau_*^{+}=\mathcal{F}^{+}(\tau_*^{+})$, where the non-lineal functional $\mathcal{F}^{+}=\mathcal{F}^{+}_\e$ is defined as:
		\[
		\mathcal{F}^{+}(\tau)(t_0,y_0)=\frac{2(\Pdt(t_0+\tmt_0(t_0,y_0;\e)+\e \tau(t_0,y_0;\e))-\Pdt(t_0))}{(1-a_0\e)(\tmt_0+\e \tau(t_0,y_0;\e))}.
		\]
We introduce the Banach space $(\mathcal{X},\normnind{\cdot}{\mathcal{D_0}^{+}})$ defined by:
		\[
		\mathcal{X}=\{\tau:\overline{\mathcal{D}_0^{+}}\to\mathbb{C},\,\mbox{$f$ real analytic in $\mathcal{D}^{+}_0$ and bounded up to the boundary},\, \normnind{\tau}{\mathcal{D}_0^{+}}<+\infty\},
		\]
where $\normnind{\tau}{\mathcal{D}_0^{+}}=\displaystyle\sup_{(t_0,y_0)\in\overline{\mathcal{D}^{+}_0}}\left\{\left|y_0\,\tau(t_0,y_0)\right|\right\}$.
                We denote as $\mathcal{B}_{32\tilde p}\subset\mathcal{X}$ the closed ball of center zero and radius $32\tilde p$ in terms of the norm $\normnind{\cdot}{\mathcal{D}^{+}_0}$. Since $\mathcal{B}_{32\tilde p}$ is a closed set of $(\mathcal{X},\normnind{\cdot}{\mathcal{D}^{+}})$, it constitutes a complete metric space in terms of the distance induced by the considered norm. We are going to prove that if $\e>0$ is as in the statement, then $\mathcal{F}^{+}:\mathcal{B}_{32\tilde p}\to\mathcal{B}_{32\tilde p}$ is a contraction. In this way, using the fixed point theorem, we establish the well defined character of $\tau_*^{+}$ in $\mathcal{D}^+_0$ as well as the bound $|\tau_*^{+}(t_0,y_0;\e)|\leq \dfrac{32\tilde{p}}{|y_0|}$. To achieve this purpose, first we show that that $\mathcal{F}^{+}(\tau)$ is well defined and belongs to $\mathcal{X}$, for every $\tau\in\mathcal{B}_{32\tilde{p}}$. First of all, we notice that if $ \tau\in\mathcal{B}_{32\tilde{p}}$, $(t_0,y_0)\in\mathcal{D}^+_0$ and $0<\e\leq\e^+_*$, then (we omit the dependence of $\tau$ and $\tau_0^+$ on $t_0$, $y_0$ and $\e$ unless necessary)
	\begin{align*}
			|{\rm Im}(t_0+\tmt_0+\e \tau)|
			&=\left|{\rm Im}(t_0)+\frac{2}{1-a_0\e}({\rm Im}(y_0)+\e{\rm Im}(\Put(t_0)))+\e{\rm Im}(\tau)\right|\\
			&\leq \robzm+4\rotzm+4\e\tilde{p}+32\e\tilde{p}\\
			&\leq \overline{\rho}_0,
		\end{align*}
		where we are bounding the size of the imaginary parts of $\Put$ and $\tau$ by its norms, and using that since $y_0^*\geq 1$ then $|\tau(t_0,y_0)|\leq 32\tilde{p}$. This means that the composition $ \Pdt(t_0+\tmt_0+\e \tau) $ is well defined and also analytic. Moreover, we also have that:
		\[
		|\tmt_0+\e \t|\geq |\tmt_0|-\e|\t|\geq \frac{|y_0|}{2}-\e\frac{|y_0\t|}{|y_0|}\geq\frac{|y_0|}{2}\left(1-\frac{64\e\tilde{p}}{y_0^{*}}\right) \geq \frac{|y_0|}{4}\geq\frac{1}{4}.
		\]
		Computations above guarantee the well defined character of $\mathcal{F}^+$.
		Next, we compute
		\[
		\Ffm(0)=\frac{2(\Pdt(t_0+\tmt_0)-\Pdt(t_0))}{(1-a_0\e)\tmt_0}.
		\]
Using bounds~\eqref{taumaiszeroandast} on $\tmt_0$, we have that
		\begin{equation*}
			|\Ffm(0)(t_0,y_0;\e)|\leq 16\frac{\tilde{p}}{|y_0|}.
		\end{equation*}
		Therefore, $ \normnind{\Ffm(0)}{\Dfm_0}<16\tilde{p}$ and, consequently, $\Ffm(0)\in\mathcal{B}_{32\rot}$. Now, we take $\t,\tb\in \mathcal{B}_{32\rot}$. Then:
		\begin{align*}
			|\Ffm(\t)-\Ffm(\tb)|=&
			\left|
			\frac{2(\Pdt(t_0+\tmt_0+\e \t)-\Pdt(t_0))}{(1-a_0\e)(\tmt_0+\e \t)}-\frac{2(\Pdt(t_0+\tmt_0+\e \tb)-\Pdt(t_0))}{(1-a_0\e)(\tmt_0+\e \tb)}
			\right|\\
			\leq &
			4\left|\frac{\Pdt(t_0+\tmt_0+\e \t)-\Pdt(t_0+\tmt_0+\e \tb)}{\tmt_0+\e \t}
			\right| \\
			&
			+4\e\left|
			\dfrac{(\Pdt(t_0+\tmt_0+\e \tb)-\Pdt(t_0))(\t-\tb)}{(\tmt_0+\e \t)(\tmt_0+\e \tb)}
			\right|.
		\end{align*}
		Thus, it follows from the Mean Value Theorem that 
		\begin{align*}
			\left|\frac{\Pdt(t_0+\tmt_0+\e \t)-\Pdt(t_0+\tmt_0+\e \tb)}{\tmt_0+\e \t}\right|\leq 4\e\tilde{p}|\t-\tb|.
		\end{align*}
                	On the other hand
		\[
		\left| \dfrac{(\Pdt(t_0+\tmt_0+\e \tb)-\Pdt(t_0))(\t-\tb)}{(\tmt_0+\e \t)(\tmt_0+\e \tb)} \right|
		\leq 32\tilde{p}|\t-\tb|.
		\]
		Then, taking into account both bounds and knowing that $288\e\tilde{p}\leq 1 $, it follows that
		\begin{align*}
			\normnind{\Ffm(\t)-\Ffm(\tb)}{\Dfm_0}\leq  144\e\tilde{p}\normnind{\t-\tb}{\Dfm_0}\leq \frac{1}{2}\normnind{\t-\tb}{\Dfm_0},
		\end{align*}
		which means that $\Ffm:\mathcal{B}_{32\tilde{p}}\to\mathcal{X}$ is a contraction. Finally, if $ \t\in\mathcal{B}_{32\tilde{p}} $, then
		\begin{align*}
			\normnind{\Ffm(\t)}{\Dfm_0}&\leq \normnind{\Ffm(\t)-\Ffm(0)+\Ffm(0)}{\Dfm_0}\\
			&\leq \normnind{\Ffm(\t)-\Ffm(0)}{\Dfm_0}+\normnind{\Ffm(0)}{\Dfm}\\
			&\leq \frac{1}{2} \normnind{\t}{\Dfm_0}+16\tilde{p}\\
			&\leq 32\tilde{p},
		\end{align*}
		which means that $\Ffm:\mathcal{B}_{32\tilde{p}}\to\mathcal{B}_{32\tilde{p}}$ is well defined and it is also a contraction as wanted.
Consequently, $\Ffm$ has a unique fixed point $ \tau^{+}_* \in\mathcal{B}_{32\tilde{p}}$ and, therefore, it satisfies bounds in~\eqref{eq:C*+}.

To obtain the bounds for the first and second order partial derivatives of $\tmt_{*}$, with respect $t_0$ and $y_0$, we apply standard Cauchy estimates to the function $\etap(t_0,y_0; \e)=y_0\tmest(t_0,y_0;\e)$. We use that $|\etap|\leq  32\tilde{p}$ in $\Dfm_0=\mathcal{D}^{+}(y_0^{*},\overline{ \rho}_0^{+},\tilde{ \rho}_0^{+})$ and we shrink the complex width of the domain by $\hat\rho_0^+$ when performing each derivative. In this way, it is not difficult to see that there is a constant $C_*^{+}= C_*^{+}(\tilde p,\hat\rho_0^+)>0$ for which equation~\eqref{eq:C*+} holds.
                \end{proof}
                We can analogously discuss the asymptotic behaviour of the time for impact of the negative solutions $\tau^{-}=\tau^{-}(t_1,y_1;\e)$. From expressions in~\eqref{eq:x<0} for the integration of~\eqref{H1} for $x<0$, it verifies the equation:
	\[
		\tau^{-}y_1+\dfrac{1+a_0\e}{2}(\tau^{-})^2-\e\tau^{-}\tilde{P}_1(t_1)+\e\tilde{P}_2(t_1+\tau^{-})-\e\tilde{P}_2(t_1)=0,
	\]
which motivates to write, if $y_1\ll 0$,
        \begin{equation}\label{taumenos}
  \tau^{-}(t_1,y_1;\e)=\tau^{-}_0(t_1,y_1;\e)+\e \tau_*^{-}(t_1,y_1;\e),
  \quad
  \tau^{-}_0(t_1,y_1;\e)=-\frac{2(y_1-\e\tilde{P}_1(t_1))}{1+a_0\e}.
\end{equation}
        Estimates concerning $\tau^{-}_0(t_1,y_1;\e)$ and $\tau_*^{-}(t_1,y_1;\e)$ are stated in Lemma~\ref{lemataumenos}, whose proof is completely analogous to that of lemma~\ref{lemataup}.
   \begin{lemma} \label{lemataumenos}
Let $\tilde{P}_1$ and $ \tilde{P}_2$ be as in lemma~\ref{lemataup}. Given $ y_1^{*} \geq 1$, $0<\overline{\rho}_1\leq\rho$, $\overline{ \rho}_1^{-}>0 $, and $ \tilde{\rho}_1^{-} >0$, such that $\overline{\rho}_1^{-}+4  \tilde{\rho}_1^{-} <\overline{\rho}_1$, we consider the set $\mathcal{D}^{-}_1=\mathcal{D}^{+}(y_1^{*},\overline{ \rho}_1^{-},\tilde{ \rho}_1^{-})$ (see~\eqref{eq:calD-}) and define $\e_*^-$ as
           \[
	     \e^{-}_{*}=\sup\left\{\e>0:  |a_0\e|\leq \dfrac{1}{2}\quad\mbox{and}\quad \overline{\rho}_1^{-}+4\tilde{\rho}_1^{-}+288\e\tilde{p}<\overline{\rho}_1 \right\}.
           \]
           Then, for every $ 0<\e\leq  \e^{-}_{*}$, the first impact time $\tau^{-}$ can be written as in~\eqref{taumenos}, with $\tau_0^{-}(t_1,y_1;\e) $ and $  \tau_*^{-}(t_1,y_1;\e) $ being real analytic functions for $(t_1,y_1)\in\mathcal{D}^{-}_1$, $2\pi$-periodic in
           $t_1$, and satisfying:
           \[
           \frac{|y_1|}{2}\leq |\tau^{-}_1(t_1,y_1;\e)|\leq 6|y_1|,\qquad 
 |\tau_*^{-}(t_1,y_1;\e)|\leq \dfrac{32\tilde{p}}{|y_1|}
 \]
 Besides that, given any $0<\hat\rho_1^-<\min\{\overline{\rho}_1^{-},\tilde{ \rho}_1^{-}\}/2$, there exists $ C_*^{-}= C_*^{-}(\tilde p,\hat\rho_1^-)>0$ such that 
		\[
			\left|\partial^{i}_{t_1}\partial^{j}_{y_1}\tau_*^{-}(t_1,y_1;\e)\right|\leq \frac{C_*^{-}}{|y_1|},
		\]
		for every $ (t_1,y_1)\in\mathcal{D}^{-}(y_1^{*}+2\hat\rho_1^-,\overline{ \rho}_1^{-}-2\hat\rho_1^-,\tilde{ \rho}_1^{-}-2\hat\rho_1^-)$, $0<\e\leq\e^{-}_{*}$, and $ 0\leq i+j\leq 2 $.
   \end{lemma}

   \noindent{\bf Expressions of Impact Map $\mathcal{P}_{\e}$.}
   We use the asymptotic expressions for the impact times $ \tmt $ and $ \tme $ obtained in lemmas~\ref{lemataup} and~\ref{lemataumenos} to provide formulas for the half positive impact map $ \mathcal{P}_{\e}^{+}(t_0,y_0)$ and the half negative impact map $ \mathcal{P}_{\e}^{-}(t_1,y_1)$ which make apparent the dominant terms of both maps when $y_0\gg 0$ and $y_1\ll 0$, respectively. Then, by composition of both maps, we obtain the corresponding expression $\mathcal{P}_{\e}= \mathcal{P}_{\e}^{-}\circ\mathcal{P}_{\e}^{+}(t_0,y_0)$.

   Let us denote $(t_1(t_0,y_0;\e),y_1(t_0,y_0;\e))= \mathcal{P}_{\e}^{+}(t_0,y_0)$. Using formulas in~\eqref{pmais},  \eqref{eq:x>0}, \eqref{eq:tau+0}, and~\eqref{taumais}, we obtain:
   \begin{equation}\label{eq:t1}
   t_1(t_0,y_0;\e)=t_0+\tmt(t_0,y_0;\e)=t_0+\frac{2(y_0-\e\tilde{P}_1(t_0))}{1-a_0\e}+\e \tmest(t_0,y_0;\e),
   \end{equation}
and also using~\eqref{eq:P1P2}:
\begin{align}
  y_1(t_0,y_0;\e)&=y_0- \tmt(t_0,y_0;\e) +\e P_1(\tmt(t_0,y_0;\e),t_0)   \nonumber
  \\
	&=-y_0+\e\;(\Put(t_1(t_0,y_0;\e))+\Put(t_0)-(1-a_0\e)\tmest(t_0,y_0;\e)). \label{eq:y1}
\end{align}
Let us denote $(\tbo(t_0,y_0;\e),\ybo(t_0,y_0;\e))=\mathcal{P}_{\e}(t_0,y_0)$. By using definitions above, we have that $(\tbo(t_0,y_0;\e),\ybo(t_0,y_0;\e))=\mathcal{P}_{\e}^{-}(t_1(t_0,y_0;\e),y_1(t_0,y_0;\e))$. We combine computations and formulas above with those in~\eqref{pmenos}, \eqref{eq:x<0}, and~\eqref{taumenos} to obtain expressions below for $\tbo(t_0,y_0;\e)$ and $\ybo(t_0,y_0;\e)$. Although in these computations $t_1= t_1(t_0,y_0;\e)$ and $y_1=y_1(t_0,y_0;\e)$ through formulas in~\eqref{eq:t1} and~\eqref{eq:y1}, we omit their dependence on $t_0$, $y_0$, and $\e$. 
We have:
\begin{align}
	\tbo(t_0,y_0;\e)&=t_1+ \tme(t_1,y_1;\e) \nonumber\\
	&=t_0+\frac{2(y_0-\e\tilde{P}_1(t_0))}{1-a_0\e}+\e \tmest(t_0,y_0;\e) -\frac{2(y_1-\e\Put(t_1))}{1+a_0\e}+\e\tau_*^{-}(t_1,y_1;\e) \nonumber \\
	&=t_0+\frac{4y_0}{1-a_0^{2}\e^{2}}+\e\left(-\frac{4\Put(t_0)}{1-a_0^{2}\e^{2}}+\frac{3-a_0\e}{1+a_0\e}\tmest(t_0,y_0;\e)+\tmeest(t_1,y_1;\e)\right),\label{eq:tbo}
\end{align}
and (we also omit the dependence of $\tbo=\tbo(t_0,y_0;\e)$ on $t_0$, $y_0$, and $\e$) using again~\eqref{eq:P1P2} and~\eqref{eq:x<0}, and formula~\eqref{taumenos} for $ \tau^{-} $,
\begin{align*}
  \ybo(t_0,y_0;\e)&=y_1+\tau^-(t_1,y_1;\e)+\e P_1(\tau^-(t_1,y_1;\e),t_1)\\
  &=  y_1+ (1+a_0\e)\tme(t_1,y_1;\e)+\e(\Put(\tbo)-\Put(t_1))\\
  &=-y_1+\e\left(\Put(\tbo)+\Put(t_1)+(1+a_0\e)\tmeest(t_1,y_1;\e)\right)\\
	&=y_0+\e\left(\Put(\tbo)-\Put(t_0)+(1-a_0\e)\tmest(t_0,y_0;\e)+(1+a_0\e)\tmeest(t_1,y_1;\e)\right).
\end{align*}
Consequently, the impact map in coordinates $(t_0,y_0)$ takes the form:
\begin{equation*}
	\mathcal{P}_{\e}	:\left\{\begin{aligned}
		\tbo&=t_0+ \alpha_{\e}(y_0)+\e f_{t_0}(t_0,y_0;\e),\\
		\ybo&=y_0+\e f_{y_0}(t_0,y_0;\e),
	\end{aligned}\right.
\end{equation*}
with
\begin{equation}\label{eq:alphaep}
  \alpha_{\e}(y_0)=\frac{4y_0}{1-a_0^{2}\e^{2}}, 
\end{equation}
\begin{equation}\label{ft0}
	f_{t_0}(t_0,y_0;\e)=-\frac{4\Put(t_0)}{1-a_0^{2}\e^{2}}+\frac{3-a_0\e}{1+a_0\e}\tmest(t_0,y_0;\e)+\tmeest(t_1(t_0,y_0;\e),y_1(t_0,y_0;\e);\e),
\end{equation}
and
\begin{align}
  f_{y_0}(t_0,y_0;\e)=&\Put(\tbo(t_0,y_0;\e))-\Put(t_0)+(1-a_0\e)\tmest(t_0,y_0,\e) \nonumber \\
 & +(1+a_0\e)\tmeest(t_1(t_0,y_0,\e),y_1(t_0,y_0,\e);\e),\label{fy0}
\end{align}
where $t_1(t_0,y_0;\e)$, $y_1(t_0,y_0;\e)$ are given by~\eqref{eq:t1}, \eqref{eq:y1}, respectively, and $\tbo=\tbo(t_0,y_0;\e)$ is given by the first component of~\eqref{eq:Pe}.

\noindent{\bf Estimates on the Impact Map $\mathcal{P}_{\e}$ (Proof of Proposition~\ref{P}).}
The most important aspect to be addressed in order to provide quantitative estimates for $\mathcal{P}_{\e}$ is to specify a complex domain for $(t_0,y_0)$ in which we can ensure that the map is well defined and in which we can use the estimates on the impact times $\tau^+$ and $ \tau^-$ provided by Lemmas~\ref{lemataup} and~\ref{lemataumenos}, respectively. This is stated precisely in Proposition~\ref{P} and below we give the details of the proof.

  With the same notations and hypotheses of Lemma~\ref{lemataup}, we select quantities below to apply Lemmas~\ref{lemataup} and~\ref{lemataumenos}, respectively.\\
\noindent $\bullet$ We set $y_0^{*}=3$, $\overline{\rho}_0=\rho$, $\overline{ \rho}_0^{+}=\rho/8$, $\tilde{\rho}_0^{+} =\rho/48$ and $\hat\rho_0^+=\rho/192$ in Lemma~\ref{lemataup}. Then, the value $\e^{+}_{*}>0$ on the statement of the lemma is defined by the conditions $|a_0\e|\leq 1/2$ and $288\e\tilde{p}<19\rho/24$. Consequently, if $C_*^{+}= C_*^{+}(\tilde p,\hat\rho_0^+)>0$ is the constant provided by Lemma~\ref{lemataup}, the following estimates hold if $ (t_0,y_0)\in\mathcal{D}^{+}_0:=\mathcal{D}^{+}(4,11\rho/96,\rho/96)$,  $0<\e\leq\e^{+}_{*}$, and $ 0\leq i+j\leq 2 $:
    \[
     \frac{|y_0|}{2}\leq |\tau^{+}_0(t_0,y_0;\e)|\leq 6|y_0|,
     \qquad
     |\tau_*^{+}(t_0,y_0;\e)|\leq \dfrac{32\tilde{p}}{|y_0|},
     \qquad
\left|\partial^{i}_{t_0}\partial^{j}_{y_0}\tau_*^{+}(t_0,y_0;\e)\right|\leq \frac{C_*^{+}}{|y_0|}.
\]
We observe that in this case $\mathcal{D}^{+}_0\subset\mathcal{D}^{+}(y_0^{*}+2\hat\rho_0^+,\overline{ \rho}_0^{+}-2\hat\rho_0^+,\tilde{ \rho}_0^{+}-2\hat\rho_0^+)$.\\ 
\noindent $\bullet$ We set $y_1^{*}=1$, $\overline{\rho}_1=\rho$, $\overline{ \rho}_1^{-}=\rho/3$, $\tilde{\rho}_1^{-} =\rho/12$ and $\hat\rho_1^-=\rho/48$ in Lemma~\ref{lemataumenos}. Then, the value $\e^{-}_{*}>0$ on the statement of the lemma is defined by the conditions $|a_0\e|\leq 1/2$ and $288\e\tilde{p}<\rho/3$. Consequently, if $C_*^{-}= C_*^{-}(\tilde p,\hat\rho_1^-)>0$ is the constant provided by Lemma~\ref{lemataumenos}, the following estimates hold if $ (t_1,y_1)\in\mathcal{D}^{-}_0:=\mathcal{D}^{-}(2,7\rho/24,\rho/24)$,  $0<\e\leq\e^{-}_{*}$, and $ 0\leq i+j\leq 2$:
    \[
     \frac{|y_1|}{2}\leq |\tau^{-}_0(t_1,y_1;\e)|\leq 6|y_1|,
     \qquad
     |\tau_*^{-}(t_1,y_1;\e)|\leq \dfrac{32\tilde{p}}{|y_1|},
     \qquad
\left|\partial^{i}_{t_1}\partial^{j}_{y_1}\tau_*^{-}(t_1,y_1;\e)\right|\leq \frac{C_*^{-}}{|y_1|}.
    \]
    Now, what we observe is that $\mathcal{D}^{-}_0\subset\mathcal{D}^{-}(y_1^{*}+2\hat\rho_1^-,\overline{\rho}_1^{-}-2\hat\rho_1^-,\tilde{ \rho}_1^{-}-2\hat\rho_1^-)$.

    We set $\overline{\rho}=11\rho/96$ and $\tilde{\rho}=\rho/96$ as the values in the statement of Proposition~\ref{P}, so that we have $\mathcal{D}^{+}(4,\overline{\rho},\tilde{\rho})=\mathcal{D}^{+}_0$. The value $\e=\e^*_{\mathcal{P}}$ of the proposition is defined by the conditions $|a_0\e|\leq 1/2$ and $864\e\tilde{p}<\rho$ which, in particular, guarantee that $\e^*_{\mathcal{P}}<\min\{\e^{+}_{*},\e^{-}_{*}\}$. For any $(t_0,y_0)\in\mathcal{D}^{+}_0$ and $0<\e<\e^*_{\mathcal{P}}$, we have (we are using~\eqref{eq:t1}, \eqref{eq:y1}, estimates above provided by Lemmas~\ref{lemataup} and~\ref{lemataumenos} and the standard trick of bounding the size of the imaginary part of a complex number by its modulus when needed):
    \begin{align*}
      |{\rm Im}(t_1(t_0,y_0;\e))| \leq &
      |{\rm Im}(t_0)|+4\left(|{\rm Im}(y_0)|+\e\|\tilde{P}_1\|_\rho\right)+\frac{32\tilde{p}\e}{|y_0|}
      <\frac{15}{96}\rho+12\tilde{p}\e\leq \dfrac{7}{24}\rho,\\
      |{\rm Im}(y_1(t_0,y_0;\e))|\leq & |{\rm Im}(y_0)|+\e\left(2\|\tilde{P}_1\|_\rho+\frac{3}{2}\frac{32\tilde{p}\e}{|y_0|}\right)<\frac{\rho}{96}+14\tilde{p}\e\leq\frac{\rho}{24},
    \end{align*}
    where we are bounding  $|\Put(t_1(t_0,y_0;\e))|\leq \|\tilde{P}_1\|_\rho$, since the first computation guarantees that $|{\rm Im}(t_1(t_0,y_0;\e))|<\rho$. Moreover:
    \[
    |y_1(t_0,y_0;\e)|\geq |y_0|-\e\left(2\|\tilde{P}_1\|_\rho+\frac{3}{2}\frac{32\tilde{p}\e}{|y_0|}\right)>|y_0|\left(1-\frac{7}{2}\tilde{p}\e\right)\geq \frac{|y_0|}{2}
    \]
    and, in particular, $|y_1(t_0,y_0;\e)|\geq 2$. These estimates guarantee that $(t_1(t_0,y_0;\e),y_1(t_0,y_0;\e))\in\mathcal{D}^{-}_0$. In addition, we also consider~\eqref{eq:tbo} (recall that we are denoting $t_1=t_1(t_0,y_0;\e)$ and $y_1=y_1(t_0,y_0;\e)$):
    \begin{align*}
 |{\rm Im}(\tbo(t_0,y_0;\e))| \leq & |{\rm Im}(t_0)|+\frac{16}{3}|{\rm Im}(y_0)|+\e\left(\frac{16}{3}\tilde{p}+7\frac{32\tilde{p}}{|y_0|}+\frac{32\tilde{p}}{|y_1|}\right)<\frac{49}{288}\rho+\frac{232}{3}\tilde{p}\e\leq \rho.
\end{align*}
    Consequently, all the expressions and compositions of functions involved in formulas~\eqref{ft0} and~\eqref{fy0} are well defined within the selected ranges for $t_0$, $y_0$ and $\e$. So, the map $\mathcal{P}_{\e}(t_0,y_0;\e)$ is analytic in $\mathcal{D}^{+}_0=\mathcal{D}^{+}(4,\ov{\rho},\tilde{\rho})$, for any $0<\e<\e^*_{\mathcal{P}}$.

    To finish the proof of Proposition~\ref{P}, it only remains to establish the existence of a constant $\tilde{C}=\tilde{C}(a_0,\rho,\tilde{p})$ for which the bounds on the statement are fulfilled. To do this, we proceed analogously as we have done to define $c^*_1$ in the proof of the Iterative Lemma~\ref{iterativelemma} of the KAM part, i.e., we are going to redefine recursively an increasing value for $\tilde{C}$ in order to meet a finite list of bounds. The value of $\tilde{C}$ in the statement of the proposition is the final value for $\tilde{C}$ thus obtained. We have checked that all the compositions involved in the definition of $\mathcal{P}_\e$ by means of formulas from~\eqref{eq:t1} to~\eqref{fy0} are well defined. Therefore, all the bounds we have previously established throughout the proof of Proposition~\ref{P} apply to them. Consequently, we can bound all the objects involved by computing their derivatives by the chain rule and then applying such bounds to them. As an example, by computing the dertivatives of $t_1(t_0,y_0;\e)$ from~\eqref{eq:t1}, we obtain the following expressions (we omit most of the dependencies in $t_0$, $y_0$ and $\e$):
    \begin{align*}
      \partial_{t_0}t_1 & =  1-\frac{\e}{1-a_0\e}\tilde{P}_1'(t_0)+\e\partial_{t_0}\tau_*^+,\quad
      \partial_{y_0}t_1=\frac{2}{1-a_0\e}+\e\partial_{y_0}\tau_*^+,\\
      \partial^2_{t_0}t_1 & =  -\frac{\e}{1-a_0\e}\tilde{P}_1''(t_0)+\e\partial^2_{t_0}\tau_*^+,\quad
      \partial_{t_0}\partial_{y_0}t_1=\e\partial_{t_0}\partial_{y_0}\tau_*^+\quad
      \partial^2_{y_0}t_1=\e\partial^2_{y_0}\tau_*^+.
\end{align*}
       Hence, we can easily establish the following estimates for them if $(t_0,y_0)\in\mathcal{D}^{+}_0$ and $0<\e<\e^*_{\mathcal{P}}$ (for certain constant $\tilde C$ defined recursively): 
    \[
|\partial_{t_0}t_1|\leq \tilde C, \quad  |\partial_{y_0}t_1|\leq \tilde C, \quad |\partial^2_{t_0}t_1|\leq \tilde C\e,  \quad   |\partial_{t_0}\partial_{y_0}t_1|\leq \frac{\tilde C\e}{|y_0|}, \quad  |\partial^2_{y_0}t_1|\leq \frac{\tilde C\e}{|y_0|}.
\]
By proceeding similarly with formula~\eqref{eq:y1}, we obtain analogous bounds for the partial derivatives of $y_1=y_1(t_0,y_0;\e)$, just by replacing above $t_1$ by $y_1$. 

Finally, we deal with the first and second order derivatives with respect to $t_0$ and $y_0$ of $f_{t_0}$ (see~\eqref{ft0}), $\bar{t}_0$ (see~\eqref{eq:Pe} and~\eqref{eq:alphaep}), and $f_{y_0}$ (see~\eqref{fy0}). We express them in terms of the the derivatives of $t_1$, $y_1$, $\tau_*^+$, and $\tau_*^-$, and bound them using all the estimates obtained previously throughout the proof of the proposition (in particular, we recall that $|y_1|\geq |y_0|/2$). Then, we obtain the following estimates in the considered domains for $t_0$, $y_0$ and $\e$:
    \begin{align*}
      |f_{t_0}| &\leq \tilde C, \quad |\partial_{t_0}f_{t_0}|\leq \tilde C, \quad  |\partial_{y_0}f_{t_0}|\leq \frac{\tilde C}{|y_0|},\quad  |\partial^2_{t_0}f_{t_0}|\leq \tilde C,  \quad   |\partial_{t_0}\partial_{y_0}f_{t_0}|\leq \frac{\tilde C}{|y_0|}, \quad  |\partial^2_{y_0}f_{t_0}|\leq \frac{\tilde C}{|y_0|},\\
      |\partial_{t_0}\bar t_0| &\leq \tilde C, \quad  |\partial_{y_0}\bar t_0|\leq \tilde C,\quad  |\partial^2_{t_0}\bar t_0|\leq \tilde C\e,  \quad   |\partial_{t_0}\partial_{y_0}\bar t_0|\leq \frac{\tilde C\e}{|y_0|}, \quad  |\partial^2_{y_0}\bar t_0|\leq \frac{\tilde C\e}{|y_0|},\\
 |f_{y_0}| &\leq \tilde C, \quad |\partial_{t_0}f_{y_0}|\leq \tilde C, \quad  |\partial_{y_0}f_{y_0}|\leq \tilde C,\quad  |\partial^2_{t_0}f_{y_0}|\leq \tilde C,  \quad   |\partial_{t_0}\partial_{y_0}f_{y_0}|\leq\tilde C, \quad  |\partial^2_{y_0}f_{y_0}|\leq\tilde C.
    \end{align*}
We omit further details because they are not difficult, but cumbersome.
                  
\section{The Localized Impact Map: Proof of Proposition~\ref{P1}}\label{propP1}
As noted above, many dependencies on $\e$ are omitted as it is set fixed throughout the proof.

\noindent{\bf Exactness of $F$ with respect to $I\d\phi$.}
We first show that the map $F(\phi,I)=(\bar{\psi}^{-1}\circ\bar{\mathcal{P}}_{\e}\circ\bar{\psi})(\phi,I)$ generated by applying the conjugation $\bar{\psi}$ of~\eqref{eq:phiI} to the exact symplectic map $\bar{\mathcal{P}}_{\e}$ of~\eqref{ImpactAngleEnergy} is still exact symplectic. We know, by construction, that there is a function $\bar V(t_0,E_0)$, depending $2\pi$-periodically in $t_0$, such that $\bar{\mathcal{P}}_{\e}^*(E_0\d t_0)=E_0\d t_0+\d \bar{V}(t_0,E_0)$. Consequently, if we define $V(\phi,I)=\tfrac{1}{\sqrt{-\Eot}}\bar{V}(\phi, \sqrt{-\Eot}\,I)$, then we have:
\begin{align*}
  F^*(I\d\phi)&=\bar{\psi}^*\,\bar{\mathcal{P}}_{\e}^*\,(\bar{\psi}^{-1})^*(I\d\phi)=\bar{\psi}^*\,\bar{\mathcal{P}}_{\e}^*\left(\frac{1}{\sqrt{-\Eot}}\,E_0\d t_0\right)\\
  &=\frac{1}{\sqrt{-\Eot}}\bar{\psi}^*(E_0\d t_0+\d \bar{V}(t_0,E_0))=I\d\phi+\bar V(\phi,I).
  \end{align*}

\noindent{\bf Explicit expression of $F$.} Since Proposition~\ref{P} gives us a close control over the properties of the map ${\mathcal{P}}_{\e}(t_0,y_0)$ of~\eqref{eq:Pe}, it is natural to express the map $F$ of~\eqref{eq:FP1} in terms of it. Given a couple $(\phi,I)$, we denote  $(\bar t_0,\bar y_0)={\mathcal{P}}_{\e}(t_0,y_0)$, and $(\bar\phi,\bar I)=F(\phi,I)$. Then, we want to express $(\bar\phi,\bar I)$ in terms of $\phi$, $I$, $\alpha_{\e}$, $f_{t_0}$, and $f_{y_0}$. Here, we are only concerned with the formal aspects of these calculations, so we do not discuss in which domain the considered expressions are well defined. Using the relations $t_0=\phi$, $y_0=\sqrt{-\sqrt{2}\,y_0^*\,I}$ (see~\eqref{eq:y0}), $\bar\phi=\bar t_0$, and $\bar I=-\tfrac{(\bar y_0)^2}{\sqrt{2}y_0^*}$, we have:
\begin{align*}
\bar \phi=\bar t_0=t_0+ \alpha_{\e}(y_0)+\e f_{t_0}(t_0,y_0;\e)=\phi+\alpha_{\e}\left(\sqrt{-\sqrt{2}\,y_0^*\,I}\right)+\e f_{t_0}\left(\phi,\sqrt{-\sqrt{2}\,y_0^*\,I};\e\right).
  \end{align*}
Consequently, $\alpha(I)$ is as given in the statement of Proposition~\ref{P1} and
\begin{equation}\label{eq:fphi}
f_\phi(\phi,I)=\e f_{t_0}\left(\phi,\sqrt{-\sqrt{2}\,y_0^*\,I};\e\right).
\end{equation}
Furthermore,
\begin{align*}
  \bar I=&
  -\frac{(\bar y_0)^2}{\sqrt{2}y_0^*}=
  -\frac{1}{\sqrt{2}y_0^*}\left(y_0+\e f_{y_0}(t_0,y_0;\e)\right)^2\\
  =& -\frac{1}{\sqrt{2}y_0^*}\left(y_0^2+2\,\e\, y_0\, f_{y_0}(t_0,y_0;\e)+\e^2(f_{y_0}(t_0,y_0;\e))^2\right)\\
  =& I-\frac{\e}{\sqrt{2}y_0^*}\left(2\, y_0\, f_{y_0}(t_0,y_0;\e)+\e(f_{y_0}(t_0,y_0;\e))^2\right).
 \end{align*}
Consequently,
\begin{equation}\label{eq:fI}
f_I(\phi,I)=-\frac{\e}{\sqrt{2}y_0^*}\left(2\, \sqrt{-\sqrt{2}\,y_0^*\,I}\, f_{y_0}(\phi,\sqrt{-\sqrt{2}\,y_0^*\,I};\e)+\e\left(f_{y_0}(\phi,\sqrt{-\sqrt{2}\,y_0^*\,I};\e)\right)^2\right).
\end{equation}

\noindent{\bf Bounding the action of $\bar\psi$.} 
 As noted above, we can express the conjugation $\bar{\psi}$ of~\eqref{eq:phiI} as $t_0=\phi$ and $y_0=\sqrt{-\sqrt{2}\,y_0^*\,I}$. 
Now, we want to show that if $(\phi,I)\in \Delta(\overline{\rho})\times{D}(I^*_0,\tilde\rho)$ (see~\eqref{Delta-norm} and~\eqref{set:D}), then $(t_0,y_0)\in\mathcal{D}^{+}(4,\overline{\rho},\tilde{\rho})$ (see~\eqref{eq:calD+}) by establishing some quantitative bounds on the obtained value for $y_0$. Clearly, for $t_0$ this assertion is obvious. Now let us denote $ I=\Iot+\Delta I $, with $y_0^ \ast=-\sqrt{2}I^\ast$ and $| \Delta I |<\tilde{ \rho}  $, which means that $ I\in {D}(I^*_0,\tilde\rho) $. Then, we have:
\begin{align*}
	y_0=y_0^*\sqrt{1-\frac{\sqrt{2}\,\Delta I}{y_0^*}}=
	y_0^*\sum_{j=0}^\infty(-1)^j\binom{1/2}{j}\left(\frac{\sqrt{2}\,\Delta I}{y_0^*}\right)^j.
\end{align*}
Since $y_0^*>5$ and $\tilde\rho<1$, we clearly have that $\tfrac{\sqrt{2}\,|\Delta I|}{y_0^*}<\tfrac{1}{2}$, meaning that the series above is well defined. 
Then, it follows :
\[
	\begin{aligned}
	|y_0-y_0^*| &\leq
	y_0^*\sum_{j=1}^\infty\left|\binom{1/2}{j}\right|\left(\frac{\sqrt{2}\,|\Delta I|}{y_0^*}\right)^j=y_0^*\left(1-\sum_{j=0}^\infty\left|\binom{1/2}{j}\right|\left(\frac{\sqrt{2}\,|\Delta I|}{y_0^*}\right)^j\right)\\
	&=
	y_0^*\left(1-\sqrt{1-\frac{\sqrt{2}\,|\Delta I|}{y_0^*}}\right)\leq |\Delta I|<\tilde\rho,
\end{aligned}
\]
where we have used that $\left|\binom{1/2}{j}\right|=(-1)^{j-1}\binom{1/2}{j}$ and that by the Mean Value Theorem we can bound $|1-\sqrt{1-x}|\leq \tfrac{\sqrt{2}|x|}{2}$ if $|x|\leq 1/2$. 
In particular, using that $y_0^\ast$ is a real number and bounding the size of the imaginary part of $y_0-y_0^\ast$ by its modulus, we conclude that
\[
|{\rm Im}(y_0)|=|{\rm Im}(y_0-y_0^*)|\leq |y_0-y_0^*| <\tilde\rho,
\]
which means that $\bar{\psi}(\Delta(\overline{\rho})\times{D}(I^*_0,\tilde\rho))\subset \mathcal{D}^{+}(4,\overline{\rho},\tilde{\rho})$, as wanted. 
Moreover, we also observe that
\[
|y_0|\geq y_0^*-|y_0-y_0^*|>y_0^*-\tilde\rho>4,
\qquad
|y_0|\leq y_0^*+\tilde\rho.
\]
Using these bounds, recalling the explicit relation between $y_0$ and $I$, given by $y_0=y_0(I)=\sqrt{-\sqrt{2}\,y_0^*\,I}$, and computing the derivatives $\partial_Iy_0=-\tfrac{\sqrt{2}}{2}\tfrac{y_0^*}{y_0}$ and $\partial_I^2y_0=-\tfrac{(y_0^*)^2}{2\,y_0^3}$, one easily obtains the estimates:
\begin{equation}\label{boundsy0}
  \frac{\sqrt{2}}{2}\frac{y_0^*}{y_0^*+\tilde\rho}\leq |\partial_Iy_0|\leq \frac{\sqrt{2}}{2}\frac{y_0^*}{y_0^*-\tilde\rho},
  \quad
  |\partial_I^2y_0|\leq \frac{(y_0^*)^2}{2\,(y_0^*-\tilde\rho)^3},
  \qquad\forall I\in{D}(I^*_0,\tilde\rho).
\end{equation}
   \\
{\bf Bounding $\alpha'(I)$, $\alpha''(I)$, $f_\phi$ and $f_I$.}
      The first consequence of~\eqref{boundsy0} are the bounds on $\alpha'(I)$ and $\alpha''(I)$ in the statement. Actually, we use that $ \alpha(I)=4(1 - a_0 ^{2}\e^{2})^{-1} y_0(I) $. Then, the bounds $\frac{5\sqrt{2}}{3}<|\alpha'(I)|<\frac{10\sqrt{2}}{3}$ and $|\alpha''(I)|\leq \frac{25}{24}$, $\forall I\in{D}(I^*_0,\tilde\rho)$, are straightforward from~\eqref{boundsy0}, by observing that since $|a_0\e|<1/2$ (see Proposition~\ref{P}) then $1 < (1 - a_0 ^{2}\e^{2})^{-1} < 4/3$, and recalling that $y_0^\ast>5$ and $\tilde\rho\in(0,1)$.

    As we have shown that $\bar{\psi}(\Delta(\overline{\rho})\times{D}(I^*_0,\tilde\rho))\subset \mathcal{D}^{+}(4,\overline{\rho},\tilde{\rho})$, then Proposition~\ref{P} allows us to conclude that expressions for $f_\phi$ and $f_I$ in~\eqref{eq:fphi} and~\eqref{eq:fI} are well defined and analytic if $(\phi,I)\in\Delta(\overline{\rho})\times{D}(I^*_0,\tilde\rho)$, provided that $0<\e<\e^*_{\mathcal{P}}$ and $y_0^*>5$. Expressions in~\eqref{eq:fphi} and~\eqref{eq:fI} also allow us to obtain explicit formulas for the first and second order derivatives of $f_\phi$ and $f_I$, with respect to $\phi$ and $I$, in terms of the derivatives of $f_{t_0}$, $f_{y_0}$ and $y_0(I)=\sqrt{-\sqrt{2}\,y_0^*\,I}$. Then, by considering bounds above on $y_0(I)$ and its derivatives as well those on $f_{t_0}$, $f_{y_0}$ and its derivatives in $\mathcal{D}^{+}(4,\overline{\rho},\tilde{\rho})$ provided by Proposition~\ref{P}, it is not difficult to conclude the existence of a constant $\bar C$ for which hold all the bounds on the statement on $f_\phi$, $f_I$ and its derivatives.    

\section{Exact sympletic structure: Proof of Proposition~\ref{ExactSympleticGeneral}}\label{sympletic}
It is well known that that the time-$T$-flow of a Hamiltonian system, the Poincaré map originating from an autonomous Hamiltonian system (at a fixed energy level) and the stroboscopic map associated to a periodic in time  Hamiltonian system are exact sympletic (see e.g.~\cite{Treschev2010}). Thus, we dedicate this section to the proof of Proposition~\ref{ExactSympleticGeneral}, which basically shows that the impact map of a smooth periodic Hamiltonian system has the exact sympletic character if it is expressed using time-energy variables.
\begin{proof}
	Let $\mathcal{H}(\xx,\yy)$ be an autonomous Hamiltonian with $n$ degrees of freedom with respect to the $2$-form $\d\xx\wedge\d\yy$. We denote by $\varphi^T(\xx_0,\yy_0)=(\xx^{\tau}(\xx_0,\yy_0), \yy^{\tau}(\xx_0,\yy_0))$ the flow associated to the system originating from  $\mathcal{H}$, where $\tau$ represents the time variable. It is a known result in the literature (see e.g.~\cite{Treschev2010}) that the time-$T$-map, for every $T\in\R, $ is exact sympletic. 
Specifically,
	\[
	\xx^{T}\d \yy^{T}=(\varphi^T)^*(\xx_0 \d \yy_0)=\xx_0 \d \yy_0+\d \mathcal{S}^{T}(\xx_0,\yy_0),
	\]
with $\varphi^T$ the time-$T$-flow of $\mathcal{H}$ and with $\mathcal{S}^{T} $ being given by
	\[
		\mathcal{S}^{T}(\xx_0,\yy_0)=\int_{0}^{T}[\langle \xx^{s}(\xx_0,\yy_0),\partial_\tau \yy^{s}(\xx_0,\yy_0)\rangle+\mathcal{H}(\varphi^s(\xx_0,\yy_0))]ds.
	\]
	If set $T=T(\xx_0,\yy_0)$ to depend on the initial conditions, we easily get the relation
	\begin{equation}\label{RelationT}
		\xx^{T(\xx_0,\yy_0)}\d(\yy^{T(\xx_0,\yy_0)})=\xx_0\d\yy_0+\d(\mathcal{S}^{T(\xx_0,\yy_0)}(\xx_0,\yy_0))-\mathcal{H}(\varphi^{T(\xx_0,\yy_0)}(\xx_0,\yy_0))\d (T(\xx_0,\yy_0)).
	\end{equation}
	We then apply the construction above to the 2-degrees of freedom autonomized Hamiltonian $\mathcal{H}(x,t,y,E)=H(x,y,t)+E$ of the statement of the proposition, being $\xx=(x,t)$ and $\yy=(y,E)$, and we select $T(x_0,t_0,y_0,E_0)>0$ as the impact time defined as $\tilde\tau(t_0,E_0)$ in~\eqref{ImpTimeXi}, but now for each initial condition $(x_0,t_0,y_0,E_0)$ close to the point $(0,t_0,y(t_0,E_0),E_0)$ of the section $\Xi$. Hence, it verifies:
	\[
	x^{T(x_0,t_0,y_0,E_0)}(x_0,t_0,y_0,E_0)=0.
	\] 
	We define the map
	\begin{align*}
		\Phi\;\colon\; &U \longrightarrow \Xi\\
		\hspace{0.3cm}	(t_0,&E_0)\longmapsto (0,t_0,y(t_0,E_0),E_0),
	\end{align*}
	and perform the pull-back $\Phi^\ast$ on equation~\eqref{RelationT} written in the context at hand. Using that $\tilde\tau(t_0,E_0)=(T\circ\Phi)(t_0,E_0)$, the definition of the impact map $(t_1,E_1)=F(t_0,E_0)$, that $\mathcal{H}=0$ along the solutions starting at $\Xi$ and that the coordinate $x$ vanishes at booth the points $\Phi(t_0,E_0)$ and the point obtained after integrating time $\tilde\tau(t_0,E_0)$ the point $\Phi(t_0,E_0)$, we conclude that
	\[
	F^\ast(t_1\d E_1)=t_0\d E_0+\d\bar{S}(t_0,E_0),
	\]
	where we are introducing
	\[
	\bar{S}(t_0,E_0)=\mathcal{S}^{\tilde\tau(t_0,E_0)}(0,t_0,y(t_0,E_0),y_0).
	\]
	By using the relation $\d(t_1\cdot E_1)=t_1\d E_1+E_1\d t_1$, we have:
	\begin{align*}
		F^\ast(E_1\d t_1) &=\d(F^\ast(t_1\cdot E_1))-F^\ast(t_1\d E_1)=
		\d(F^\ast(t_1\cdot E_1))-t_0\d E_0-\d\bar{S}(t_0,E_0)\\
		& =E_0\d t_0+\d S(t_0,E_0),
	\end{align*}
	where
	\[
	S(t_0,E_0)=F^\ast(t_1\cdot E_1)-t_0\cdot E_0-\bar{S}(t_0,E_0).
	\]
	We observe that using that $t^s(\Phi(t_0,E_0))=t_0+s$ but that $x^s(\Phi(t_0,E_0))$, $y^s(\Phi(t_0,E_0))$, and $E^s(\Phi(t_0,E_0))$ depend on $t_0$ in a $2\pi$-periodic way, we can write
	\[
	\bar{S}(t_0,E_0)=\int_0^{\tilde\tau(t_0,E_0)}t_0\cdot\partial_{\tau}E^s(\Phi(t_0,E_0))\,\d s+\tilde{S}(t_0,E_0)=t_0\cdot\left(f_{E_0}(t_0,E_0)-E_0\right)+\tilde{S}(t_0,E_0),
	\]
	where $\tilde{S}(t_0,E_0)$ is $2\pi$-periodic in $t_0$. Then, using that $F^\ast(t_1\cdot E_1)=(t_0+f_{t_0}(t_0,E_0))\cdot f_{E_0}(t_0,E_0)$, we conclude that
	\[
	{S}(t_0,E_0)=f_{t_0}(t_0,E_0)\cdot f_{E_0}(t_0,E_0)-\tilde{S}(t_0,E_0),
	\]
	is $2\pi$-periodic in $t_0$, as claimed in the statement.
\end{proof}

\section*{Competing interests} To the best of our knowledge, no conflict of interest, financial or other, exists.

\section*{Authors' contributions} All persons who meet authorship criteria are listed as authors, and all authors certify that they have participated sufficiently in the work to take public responsibility for the content, including participation in the conceptualization, methodology, formal analysis, investigation, writing-original draft preparation and writing-review \& editing.

\section*{Acknowledgments}
T. M-Seara and J. Villanueva are supported by the grant AGRUPS-2023: Grup de Sistemes Dinàmics de la UPC.
T. M-Seara is supported by the grant PID-2021-122954NB-100 funded by MCIN/AEI/10.13039/501100011033 and ``ERDF A way of making Europe'', the Catalan Institution for Research and Advanced Studies (via an ICREA Academia Prize 2019) and the Spanish State Research Agency, through the Severo Ochoa and Mar\'{\i}a de Maeztu Program for Centers and Units of Excellence in R\&D (CEX2020-001084-M). LVMFS is partially supported by S\~{a}o Paulo Research Foundation (FAPESP) grants 2018/22398-0 and 2021/11515-8.

\bibliographystyle{abbrv}
\bibliography{referencesJTL}

\end{document}